\documentclass{amsart}

\usepackage{graphics}

\usepackage{amsmath,amssymb,amsfonts,latexsym}

\usepackage{mathrsfs}
\usepackage[dvips]{geometry}
\usepackage{array}

\DeclareFontFamily{T1}{calligra}{}
\DeclareFontShape{T1}{calligra}{m}{n} {<-> callig15}{}

\newtheorem*{theorem-nonumber}{Theorem}{\bfseries}{\rmfamily}
\newtheorem*{maintheorem}{Main Theorem}{\bfseries}{\rmfamily}

\newtheorem{theorem}{Theorem}
\newtheorem{lemma}{Lemma}
\newtheorem{proposition}{Proposition}
\newtheorem{corollary}{Corollary}
\theoremstyle{definition}

\newtheorem{remark}{Remark}

\numberwithin{theorem}{section}
\numberwithin{proposition}{section} \numberwithin{lemma}{section}
\numberwithin{corollary}{section} \numberwithin{remark}{section}
\numberwithin{example}{section}

\newcommand{\bT}{\mathbf{Tab}}

\newcommand{\bS}{\mathbf{S}}

\newcommand{\sh}{{\rm sh\,}}

\newcommand{\ov}{\overline}

\newcommand{\B}{\mathcal{B}}

\newcommand{\F}{\mathcal{F}}

\newcommand{\Sl}{\mathfrak{s}\mathfrak{l}}

\newcommand{\Co}{\mathbb{C}}

\newcommand{\rar}{\rightarrow}

\newcommand{\st}{\subset}

\newcommand{\K}{{\mathcal K}}

\begin{document}

\title[Singularity of components of Springer fibers]
{On the singularity of the irreducible components of \\ a Springer fiber in $\Sl_n$}

\author{Lucas Fresse}
\address{Department of Mathematics,
The Weizmann Institute of Science, Rehovot 76100, Israel}
\email{lucas.fresse@weizmann.ac.il}

\author{Anna Melnikov}
\address{Department of Mathematics,
University of Haifa, Haifa 31905, Israel}
\email{melnikov@math.haifa.ac.il}
\keywords{Flag variety;  Springer fibers;
Young diagrams and tableaux}

\begin{abstract}
Let $\B_u$ be the Springer fiber over a nilpotent endomorphism
$u\in {\rm End}(\Co^n)$. Let $J(u)$ be the Jordan form of $u$
regarded as a partition of $n.$ The irreducible components of
$\B_u$ are all of the same dimension. They are labelled by Young
tableaux of shape $J(u).$ We study the question of singularity of
the components of $\B_u$ and show that all the components of
$\B_u$ are nonsingular if and only if
$J(u)\in\{(\lambda,1,1,\ldots),\ (\lambda_1,\lambda_2),\ (\lambda_1,\lambda_2,1),\ (2,2,2)\}$.
\end{abstract}

\maketitle

\section{Introduction}

\subsection{The Springer fiber $\B_u$}

\label{1.1}

Let $V=\Co^n$ for $n\geq 0$ and let $u:V\rar V$ be a nilpotent
endomorphism. Let $\B:=\B_n$ denote the variety of complete flags
of $V$ and $\B_u$ the variety of complete flags preserved by $u$,
that is flags $(V_0,\ldots,V_n)$ such that $u(V_i)\st V_i$ for all
$i\ :\ 0\leq i\leq n.$ Both $\B$ and $\B_u$ are algebraic
projective varieties. The variety $\B_u$ is called the Springer
fiber over $u$ since it can be regarded as the fiber over $u$ of
the Springer resolution of singularities of the cone of nilpotent
endomorphisms of $V$ (cf., for example \cite{S1}).

The variety $\B_u$ is reducible in general and its irreducible
components play a key role in Springer's Weyl group
representations, as well as in the study of the primitive ideals
in $U(\Sl_n(\Co)).$ Description of their geometry is both a very
important and challenging topic for more than 30 years. There are
a lot of open questions and not so many answers in this field.

Since $u$ is a nilpotent endomorphism, 0 is
its unique eigenvalue and its Jordan form $J(u)$ is completely
described by the lengths of the Jordan blocks, which can be
written as a partition
$\lambda=(\lambda_1,\lambda_2,\ldots,\lambda_k)$ of $n$ where we
order the lengths in non-increasing order, that is
$\lambda_1\geq\lambda_2\geq\ldots\geq \lambda_k>0.$   We put
$J(u):=\lambda.$ Given a partition $\lambda$, the corresponding
Young diagram $Y_\lambda$ is defined to be an array of $k$  rows
of cells starting on the left, with the $i$-th row containing
$\lambda_i$ cells. Put $Y(u):=Y_\lambda$ if $J(u)=\lambda.$

Obviously $\B_u$ depends only on $J(u)$ or equivalently on $Y(u)$
so that we can talk about $\B_u$ for $u$ in a given nilpotent
orbit (under the conjugation by $GL_n(\Co)$).

Given a Young diagram $Y_\lambda$, fill in its boxes with numbers
$1,\ldots,n$ in such a way that the entries increase in rows from
left to right and in columns from top to bottom. Such an array is
called  a Young tableau of shape $\lambda.$ Let us denote by
$\bT_\lambda$ the set of Young tableaux of shape $\lambda.$

By N. Spaltenstein \cite{Sp}, $\B_u$ is an equidimensional
variety and its components are in bijection with $\bT_{J(u)}$.

\subsection{Statement of the main result}

\label{1.2}

Up to now the singularity of the irreducible components of $\B_u$ has been studied
only in three  cases. J.A. Vargas \cite{V} and N. Spaltenstein
\cite{Sp1} described the geometry of the components in the
simplest case of $J(u)=(\lambda_1,1,1\ldots)$ (hook case) and in
particular showed that all the components are nonsingular in that
case. They were also the first to compute an example of a singular
component in the case $J(u)=(2,2,1,1).$ Twenty years later F. Fung
\cite{F} described the geometry of the components in the case of
$J(u)=(\lambda_1,\lambda_2)$ (two-row case). In this case again,
all the components are nonsingular. Finally, the first coauthor
\cite{Fr} gave a  combinatorial criterion for a component to be
singular in the case of $J(u)=(2,2,\ldots)$ (two-column case).

The aim of this paper is to give a full classification of
nilpotent orbits with respect to the singularity of the components
of $\B_u.$ The main result of the paper is
 \begin{maintheorem} {\em Let $u\in{\rm End}(\Co^n)$ be a nilpotent endomorphism.
Then all the irreducible components of $\B_u$ are nonsingular in the
following four cases:
\begin{itemize}
\item[\rm (i)] $J(u)=(\lambda_1,1,\ldots)$ \ (hook case),
\item[\rm (ii)] $J(u)=(\lambda_1,\lambda_2)$ \ (two-row case),
\item[\rm (iii)] $J(u)=(\lambda_1,\lambda_2,1)$ \ (two-row-plus-one-box case),
\item[\rm (iv)] $J(u)=(2,2,2)$ \ (an exceptional case, due to small $n$).
\end{itemize}
In all other cases $\B_u$ admits singular components.}
 \end{maintheorem}

\subsection{Construction of the irreducible components of the Springer fiber}

\label{1.3}

Let us explain Spaltenstein's construction
in some detail. The full details can be found in \cite[\S
II.5]{Sp1}.

Given a partition $\lambda=(\lambda_1,\ldots,\lambda_k)$ let
$\lambda^*=(\lambda_1^*,\lambda_2^*,\ldots,\lambda_{\lambda_1}^*)$
denote the  conjugate partition, that is the list of the lengths
of the columns in $Y_\lambda.$ For a Young tableau $T$ put
$\sh(T)$ to be its shape, that is the corresponding Young diagram.

Given a nilpotent $u\in{\rm End}(\Co^n)$ with $J(u)=\lambda$ let
$(V_0,\ldots,V_n)$ be some flag in $\B_u.$ Note that for any $i\
:\ 1\leq i\leq n$  $Y(u_{|V_i})$ differs from
$Y(u_{|V_{i-1}})$ by exactly one (corner)  box. Thus,  each flag
$(V_0,\ldots,V_n)\in \B_u$  determines  the chain of Young diagrams
$(Y(u_{|V_0}),\ldots, Y(u_{|V_{n-1}}),Y(u)).$

On the other hand for $T\in\bT_{\lambda}$ and $i\ :\ 1\leq i<n$
let $\pi_{1,i}(T)$ be the tableau obtained from $T$ by deleting the
boxes containing the numbers $i+1,\ldots,n$. Comparing
$\sh(\pi_{1,i}(T))$ and $\sh(\pi_{1,i-1}(T))$ one sees that they
differ by one (corner) box, containing $i.$ In such a way every
Young tableau can be regarded as a chain of Young diagrams. Put
$$\F_T:=\{(V_0,\ldots,V_n)\in \B_u\ :\
Y(u_{|V_i})=\sh(\pi_{1,i}(T)),\ 0\leq i\leq n\}.$$
Note that
$$\B_u=\bigsqcup\limits_{T\in\bT_\lambda}\F_T$$  is a partition of
$\B_u$. Moreover, by \cite[\S II.5]{Sp1}, $\F_T\ :\
T\in\bT_\lambda$ are smooth irreducible subvarieties of $\B_u$ and
$$\dim\F_T=\dim\B_u=\sum\limits_{i=1}^{\lambda_1}{\frac{\lambda^*_i(\lambda_i^*-1)} 2}\eqno{(1)}$$
so that $\{\K^T:=\overline\F_T\ :T\in\bT_{\lambda}\}$ are all the irreducible components of
$\B_u.$

\subsection{Outline of the proof of the main theorem}

\label{1.4}

Given $T\in\bT_\lambda$, for $i\ :\ 1\leq i\leq
n-1$, the subtableau $\pi_{1,i}(T)$ is associated to a component
$\K^{\pi_{1,i}(T)}\subset {\mathcal B}_{u'}$, for a nilpotent
$u'\in\mathrm{End}(V')$ such that
$Y(u')=\mathrm{sh}(\pi_{1,i}(T))$. The natural question is what is
the connection between singularities of $\K^T$ and
$\K^{\pi_{1,i}(T)}$. The answer is given by Theorem
\ref{theorem-2.3}:
\begin{theorem-nonumber}
{\em If the component $\K^{\pi_{1,n-1}(T)}$ is singular then the
component $\K^T$ is singular. Moreover, if $n$ lies in the last
column of $T$, then $\K^T$ is singular if and only if
$\K^{\pi_{1,n-1}(T)}$ is singular.}
\end{theorem-nonumber}

The first part of this theorem, claiming that the existing
singularity cannot disappear, is very natural. This part together
with the example of a singular component of shape $(2,2,1,1)$
shows that any nilpotent endomorphism $u$ with at least four
Jordan blocks and at least two of them of length at least 2, that
is such that
$J(u)=(\lambda_1,\lambda_2,\lambda_3,\lambda_4,\ldots)$ where
$\lambda_2\geq 2$, admits a singular component. The only open case
left is the case of a nilpotent (non-hook) orbit with three Jordan
blocks. As we show in section \ref{2.5} there exists a singular component
for the form $(3,2,2)$. Thus again by the theorem, any nilpotent
non-hook endomorphism with three blocks such that the minimal
block is of length at least 2 and the maximal block is of length
at least 3 admits a singular component.

The second (if and only if) part of the theorem is very useful.
The fact that all the components in the two-row case
$(\lambda_1,\lambda_2)$ are nonsingular is obtained in section \ref{3.7}
as its easy corollary. This simplifies drastically the original
proof of F. Fung.

Further in Section 4 we use this part of the theorem to show that
all the components in the case $(\lambda_1,\lambda_2,1)$ (which we
call two-row-plus-one-box case) are nonsingular. But here the
proof is much more involved. We first note that this question is
equivalent to the nonsingularity of the components in the case
$(r,r,1)$. Then we partition  the components of $\B_{(r,r,1)}$
into $r$ classes of equisingular ones. Finally we show that each
class admits a nonsingular component of some very special form.
The proof of nonsingularity of these special components uses the
techniques developed by the first author for computations of the
singularities of the components in the two-column case. The proof
of nonsingularity of the components in the two-row-plus-one-box
case constitutes the most technically involved part of the paper.

\medskip

The body of the paper consists of three parts. In Section 2 we
prove the theorem formulated in \ref{1.4} and provide its first
corollaries.

In Section 3 we develop combinatorial techniques of partitioning
$\B_u$ into classes of equisingular components. We prefer to call
them equinonsingular classes in the given context since our aim is
to show the nonsingularity of the components.

Finally, in Section 4 we show that in the two-row-plus-one-box
case all the components are nonsingular.

To make the paper as self-contained as possible we formulate all the
results we need in the due places.
The reader can find an index of the notation at the end of the paper.

\section{Inducing the singularity of a component}

\subsection{Preliminary observation}

\label{2.2}

Let us start with a preliminary observation. Note that, exactly in
the same way as we have defined $\pi_{1,i}(T)$ in section \ref{1.3}, we
can define $\pi_{i,j}(T)$ for $i\leq j$ to be the tableau with
entries $i,\ldots,j$ obtained from $T$ by deleting the boxes
containing $j+1,\ldots, n$ and removing the boxes $1,\ldots,i-1$
by the procedure of jeu de taquin (cf. \cite{Ful}). In particular,
$\sh(\pi_{i-1,n}(T))$ differs from $\sh(\pi_{i,n}(T))$ by exactly
one corner box obtained by removing $i-1$ by  jeu de taquin. In
such a way one can associate to the tableau $T$ a chain of Young
diagrams
$$(\sh(\pi_{n,n}(T)),\sh(\pi_{n-1,n}(T))\ldots,\sh(\pi_{1,n}(T)))=:(D_1,\ldots,D_n).$$
Put $Sch(T)$ to be the Young tableau obtained from this chain by
putting $i$ in the only new box of $D_i$ compared to $D_{i-1}$.
The tableau $Sch(T)$ is called the Sch\"utzenberger transform of
$T$ (cf. \cite{Ful}).

Respectively, since for any $F=(V_0,\ldots,V_n)\in{\mathcal B}_u$ and any $i\ :
1\leq i\leq n$ we have $u(V_i)\subset V_i$, one can consider the
action of $u$ on $(V_{n}/V_{n},\ldots,V_n/V_{0})$. Respectively
put
$$\F'_T=\{(V_0,\ldots,V_n)\in{\mathcal B}_u\ :\
Y(u|_{(V_n/V_i)})=\sh(\pi_{i+1,n}(T)),\ 0\leq i\leq n\}.$$
Note that $\F_T\ne \F'_T$. However, exactly as in the case of $\F_T$, one
has a partition
$$\B_u=\bigsqcup\limits_{T\in\bT_\lambda}\F'_T.$$
Moreover, $\F'_T\ :\ T\in\bT_\lambda$ are smooth irreducible
subvarieties of $\B_u$ of the same dimension as $\F_T$. Therefore
the irreducible components of ${\mathcal B}_u$ are also obtained
as the closures of the subsets $\F'_T$. By \cite[\S 3.4]{vL}, we
have $\overline\F'_T=\ov\F_T=\K^T$.

Fix a non-degenerate symmetric bilinear form $b(\cdot,\cdot)$ on $V$ for
which $u$ is self-adjoint, and write $W^\perp=\{v\in V:b(v,w)=0\
\forall w\in W\}$ whenever $W\subset V$ is a subspace. The
application ${\mathcal B}_u\rightarrow{\mathcal B}_u$,
$(V_0,\ldots,V_n)\mapsto (V_n^\perp,\ldots,V_0^\perp)$ maps
$\F'_T$ onto $\F_{Sch(T)}$, providing an isomorphism of algebraic
varieties $\K^T\cong \K^{Sch(T)}$.

Note that $Sch(\pi_{i,n}(T))=\pi_{1,n-i+1}(Sch(T))$ (where we
identify $\pi_{i,n}(T)$ with the Young tableau of entries
$1,\ldots,n-i+1$ obtained by replacing $j\rar j-i+1$). In
particular, the results connecting singularities of
$\K^{\pi_{1,n-1}(T)}$ and $\K^T$ can be translated into results
connecting  singularities of $\K^{\pi_{2,n}(T)}$ and $\K^T.$

\subsection{Inductive criterion of singularity}

\label{2.3}

For $T\in\bT_\lambda$ set $T':=\pi_{1,n-1}(T)$ and $Y':=\sh(T')$.
Let ${\mathcal H}_u$ be the variety of $u$-stable hyperplanes
$H\subset V$ and let ${\mathcal H}'_u\subset{\mathcal H}_u$ be the
subset of hyperplanes $H\in{\mathcal H}_u$ such that the Jordan
form of the restriction $u_{|H}\in\mathrm{End}(H)$ corresponds to
the Young diagram $Y(u_{|H})=Y'$. Fix $V'\in{\mathcal H}'_u$ and let
$u':=u_{|V'}\in\mathrm{End}(V')$. Let ${\mathcal B}_{u'}$ be the
variety of $u'$-stable complete flags on $V'$. Let $\K^{T'}\subset
{\mathcal B}_{u'}$ be the irreducible component corresponding to the
tableau $T'$. In these terms we get
\begin{theorem}
\label{theorem-2.3} If the component $\K^{T'}$ is singular, then
the component $\K^T$ is singular. Moreover, if $n$ lies in the
last column of $T$, then $\K^T$ is singular if and only if
$\K^{T'}$ is singular.
\end{theorem}

\begin{proof}
First, note that a hyperplane $H\subset V$ is $u$-stable if and
only if it contains $\mathrm{Im}\,u$. Thus, ${\mathcal H}_u$
corresponds to the variety of hyperplanes of the space
$W:=V/\mathrm{Im}\,u$. Denote it by ${\mathcal H}(W)$.

We introduce some algebraic groups. Put $P=\{g\in GL(V):g(\ker
u^k)=\ker u^k\ \forall k\geq 0\}$. This is a parabolic subgroup of
$GL(V)$. Let $\zeta:V\rightarrow W$ be the natural surjection.
Define $Q=\{g\in GL(W):g(\zeta(\ker u^k))=\zeta(\ker u^k)\ \forall
k\geq 0\}$. This is a parabolic subgroup of $GL(W)$. Let $k_n$ be
the number of the column of $T$ containing $n$. Note that
${\mathcal H}'_u$ is the set of hyperplanes $H\subset V$ such that
$H\supset \ker u^{k_n-1}$, $H\not\supset \ker u^{k_n}$, hence it
can be identified  with a $Q$-orbit of ${\mathcal
H}_u\cong{\mathcal H}(W)$. In particular, ${\mathcal H}'_u$ is a
locally closed subset in ${\mathcal H}_u$, irreducible and
nonsingular.

Define ${\mathcal U}=\{(V_0,\ldots,V_n)\in \K^T:
V_{n-1}\in{\mathcal H}'_u\}$. Then this is a locally closed subset
of the component $\K^T$. Moreover, by definition, ${\mathcal
U}\supset \F_T$. Since $\K^T=\overline{\F_T}$, it follows that
${\mathcal U}$ is actually a nonempty open subset of $\K^T$. In
addition, if $n$ lies in the last column of $T$, then $\ker
u^{k_n}=V$. In this case ${\mathcal H}'_u=\{H\in{\mathcal
H}_u:H\supset \ker u^{k_n-1}\}$, and ${\mathcal H}'_u$ is actually
a closed subset of ${\mathcal H}_u$ so that ${\mathcal U}=\K^T$.
Therefore, to prove the theorem, it is sufficient to show that
${\mathcal U}$ is nonsingular if and only if $\K^{T'}$ is
nonsingular.

To do this, we consider the map
\[\Phi:{\mathcal U}\rightarrow
{\mathcal H}'_u,\ (V_0,\ldots,V_n)\mapsto V_{n-1},\] and we show
that $\Phi$ is an algebraic fiber bundle with fiber isomorphic to
$\K^{T'}$.

Set $Z(u):=\{h\in GL(V):huh^{-1}=u\}$ to be the stabilizer of $u$.
This is a connected, closed subgroup of $GL(V)$, and it acts on
${\mathcal B}_u$, thus, leaving invariant each component. Also
${\mathcal U}$ is stable by this action. In addition $Z(u)$
naturally acts on ${\mathcal H}_u$ and ${\mathcal H}'_u$, and the
map $\Phi$ is $Z(u)$-equivariant. We have $Z(u)\subset P$ and a
natural morphism of algebraic groups $\varphi:Z(u)\rightarrow Q$,
and the action of $Z(u)$ on ${\mathcal H}_u$ commutes with the map
$\varphi$ and the action of $Q$. Actually, we can see that there
is a morphism of algebraic groups $\psi:Q\rightarrow Z(u)$ such
that $\varphi\circ\psi=id_Q$ (see for example \cite[\S 3.6]{Fr1}),
so that $Q$ can be interpreted as a subgroup of $Z(u)$. In
particular, ${\mathcal H}'_u$ is a $Z(u)$-orbit of ${\mathcal
H}_u$.

First we show that $\Phi$ is locally trivial. Let $H\in {\mathcal
H}'_u$. By Schubert decomposition, there is a Borel subgroup
$B_H\subset Q$ with a unipotent subgroup $U_H\subset B_H$ such
that the map $U_H\stackrel{\sim}{\rightarrow}{\mathcal
O}_H\subset{\mathcal H}'_u$, $g\mapsto gH$ is an open immersion.
Therefore, the map $U_H\times \Phi^{-1}(H)\rightarrow
\Phi^{-1}({\mathcal O}_H)$,
$(g,(V_0,\ldots,V_n))\mapsto(\psi(g)V_0,\ldots,\psi(g)V_n)$ is an
isomorphism of algebraic varieties, and $\Phi$ is trivial over
${\mathcal O}_H$.

It remains to show that $\Phi^{-1}(H)\cong \K^{T'}$. Since
${\mathcal U}$ is irreducible, it follows that $\Phi^{-1}(H)$ is
irreducible. Let ${\mathcal B}_{u_{|H}}$ be the variety of $u$-stable
complete flags on $H$. The natural inclusion $\Psi:{\mathcal B}_{u_{|H}}\rightarrow
{\mathcal B}_u$ is a closed immersion. We have $\Psi(\F_{T'})\subset
\F_T\cap \Phi^{-1}(H)$, hence $\Psi(\K^{T'})$ is a closed subset
of $\Phi^{-1}(H)$.

Let $\lambda_{k}^*$ be the length of the $k$-th column of $T$. By
formula $(1)$ we obtain $\dim \K^{T'}=\dim \K^T-(\lambda_{k_n}^*-1)$. Also,
$\dim W/\zeta(\ker u^{k_n-1})=\lambda_{k_n}^*$ so that $\dim
{\mathcal H}'_u=\lambda_{k_n}^*-1$. Thus, $\dim \Phi^{-1}(H)=\dim
{\mathcal U}-\dim {\mathcal H}'_u=\dim \K^{T'}$, and we get
finally $\Psi(\K^{T'})=\Phi^{-1}(H)$. The proof of the theorem is
now complete.  
\end{proof}

As a straightforward corollary of Theorem \ref{theorem-2.3} we get

\begin{corollary}
\label{corollary-2.4}
If $\K^T$ is a component such that $\K^{\pi_{i,j}(T)}$ is singular
for some $i,j \ :\ 1\leq i<j\leq n$ and $(i,j)\ne(1,n)$ then
$\K^T$ is singular.
\end{corollary}
\begin{proof}
If $\K^{\pi_{i,j}(T)}$ is singular and $(i,j)\ne(1,n)$ then either
$i>1$ or $j<n$. If $j<n$ then by induction hypothesis
$\K^{\pi_{1,n-1}(T)}$ is singular, so that, by Theorem \ref{theorem-2.3},
$\K^T$ is singular.

If $j=n$ then by induction hypothesis $\K^{\pi_{2,n}(T)}$ is
singular so that by Theorem \ref{theorem-2.3} and subsection \ref{2.2},
$\K^T$ is singular.  
\end{proof}

\subsection{Construction of singular components}

\label{2.5}

The first example of a singular component was obtained by Vargas
\cite{V} and Spaltenstein \cite{Sp1} and it is
$\K^S\in\B_{(2,2,1,1)}$ where
$$S=\begin{array}{ll}
1&3\\
2&5\\
4& \\
6& \\
\end{array}.
$$
Moreover, by \cite{Fr} this is the only singular component in the
case $(2,2,1,1).$

Let us show
\begin{proposition}
\label{proposition-singcomponent-322}
The component $\K^T\in\B_{(3,2,2)}$ associated to the tableau
$$T=\begin{array}{lll}
1&2&5\\
3&4&\\
6&7\\
\end{array}$$
is singular.
\end{proposition}
\begin{proof}
Fix a Jordan basis $(e_1,\ldots,e_7)$ of $\Co^7$ such that $u$
acts on the basis by
\[
e_7\mapsto e_4\mapsto e_1\mapsto 0 \qquad e_5\mapsto e_2\mapsto 0
\qquad e_6\mapsto e_3\mapsto 0.
\]
For $i=0,\ldots,7$ let $V_i=\mathrm{Span}\{ e_1,\ldots,e_i\}$, and
consider the flag $F_0=(V_0,\ldots,V_7)$. Let $B\subset
GL(\mathbb{C},7)$ be the subgroup of lower triangular matrices and
let $U\subset B$ be the subgroup of unipotent matrices. Let
$\Omega$ be the $B$-orbit of $F_0$ in the variety of complete
flags. The map $U\rightarrow \Omega$, $g\mapsto gF_0$ is an
isomorphism of algebraic varieties. Let ${\mathcal X}\subset
M(\mathbb{C},7)$ be the subspace of nilpotent lower triangular
matrices. Let $E_{i,j}\in M(\mathbb{C},7) $ be the standard basic
matrix
 $$(E_{i,j})_{k,l}=\left\{\begin{array}{ll}
 1&{\rm \ if}\ (k,l)=(i,j)\\
 0&{\rm \ otherwise.}\\
\end{array}\right.$$
The set $\{(E_{i,j})\}_{1\leq j<i\leq 7}$  forms a (standard)
basis of ${\mathcal X}$.
The map ${\mathcal X}\rightarrow U$,
$g\mapsto g+I_7$ is an isomorphism of algebraic varieties. Both
isomorphisms combined provide an isomorphism $\varphi:{\mathcal
X}\rightarrow \Omega$, $g\mapsto (g+I_7)F_0$, and $\varphi$
reduces to an isomorphism $\varphi^{-1}(\Omega\cap {
\K}^T)\rightarrow \Omega\cap { \K}^T$.

Consider the map $f:\mathbb{C}^6\rightarrow {\mathcal X}$ defined
by
\[f:(t_1,\ldots,t_6)\mapsto
\left(\begin{array}{ccccccc}
0 & 0 & 0 & 0 & 0 & 0 & 0 \\
t_1 & 0 & 0 & 0 & 0 & 0 & 0 \\
t_1t_2 & t_2+t_3 & 0 & 0 & 0 & 0 & 0 \\
0 & t_3t_4t_5 & t_4t_5 & 0 & 0 & 0 & 0 \\
0 & t_1t_3t_4t_5 & t_1t_4t_5 & t_4 & 0 & 0 & 0 \\
0 & t_1t_2t_3t_4t_5 & t_1t_2t_4t_5 & t_2t_4 & t_2+t_6 & 0 & 0 \\
0 & 0 & 0 & 0 & t_5t_6(t_4-t_1) & t_5(t_4-t_1) & 0 \\
\end{array}\right).\]
A straightforward computation shows that $\varphi\circ
f(t_1,t_2,t_3,t_4,t_5,t_6)\in\Omega\cap\F_T$ whenever $t_3$,
$t_4$, $t_5$, $t_6$, $t_4-t_1$ are all nonzero. Thus,
$f(\mathbb{C}^6)\subset \varphi^{-1}(\Omega\cap \K^T)$. In
particular $0=f(0)\in\varphi^{-1}(\Omega\cap { \K}^T)$.

The tangent space of $\varphi^{-1}(\Omega\cap { \K}^T)$ at $0$,
denoted by ${\mathcal T}_0\,\varphi^{-1}(\Omega\cap { \K}^T)$, can
be seen as a vector subspace of ${\mathcal X}$. For
$i=1,\ldots,6$, let $\iota_i:\mathbb{C}\rightarrow \mathbb{C}^6$,
$t\mapsto (t_1,\ldots,t_6)$ be the map defined by $t_i=t$, and
$t_j=0$ for $j\not=i$. The curve
$\{f(\iota_i(t)):t\in\mathbb{C}\}$ lies in
$\varphi^{-1}(\Omega\cap \K^T)$. Considering for each $i$ the
tangent vector at $0$ to this curve, we get:
\[E_{2,1},\ E_{3,2},\ E_{5,4},\ E_{6,5}\in {\mathcal T}_0\,\varphi^{-1}(\Omega\cap { \K}^T).\]
The curves $\{f(t,1,-1,0,0,-1):t\in\mathbb{C}\}$,
$\{f(0,1,-1,t,0,-1):t\in\mathbb{C}\}$ and
$\{f(0,0,0,t,1,0):t\in\mathbb{C}\}$ also pass through $0$ for
$t=0$. Considering the tangent vectors at $0$ to these curves, we
get:
\[E_{2,1}+E_{3,1},\ E_{5,4}+E_{6,4},\ E_{4,3}+E_{5,4}+E_{7,6}\in {\mathcal T}_0\,\varphi^{-1}(\Omega\cap { \K}^T).\]
We have constructed seven vectors of the tangent space ${\mathcal
T}_0\,\varphi^{-1}(\Omega\cap { \K}^T)$, and we see that all these
vectors are linearly independent. It follows
\[\mathrm{dim}\,{\mathcal T}_0\,\varphi^{-1}(\Omega\cap { \K}^T)\geq 7.\]
On the other hand by $(1)$  $\dim \K^T=6$ so that
\[\mathrm{dim}\,\varphi^{-1}(\Omega\cap { \K}^T)=\mathrm{dim}\,\Omega\cap { \K}^T
=\mathrm{dim}\,{ \K}^T=6.\] It follows that
$\varphi^{-1}(\Omega\cap { \K}^T)$ is singular. Thus, $\Omega\cap
{\mathcal \K}^T$ is singular. This is an open subset of ${ \K}^T$.
Therefore, the component ${ \K}^T$ is singular.  
\end{proof}

\begin{remark}
Checking other components of $\B_{(3,2,2)}$ one can see that all
the other components but ${\mathcal K}^T$ given above are
nonsingular. Note that by $(1)$  $\dim \K^S=6+1$ so that $\K^T$ is a 
singular component of a smaller dimension than
$\K^S$.

In fact, since all the components in the hook case are
nonsingular as well as the components in the two-row case and as
we show also the components in the two-row-plus-one-box case, all
the components of dimension smaller than 6 are nonsingular, so
that $\K^T$ is a singular component of the minimal
dimension. Also since all the components of $(2,2,2)$ are
nonsingular there are no singular components of dimension 6 in
$GL_n$ for $n\leq 6.$

Of course from this example one can obtain a singular
component of dimension $6$ for any $n\geq 7$ simply  by taking
$T'$ obtained from $T$ by adding
$8,9,\ldots,n$ to the first row (one has by $(1)$ that $\dim
\K^{T'}=\dim \K^T=6$ and by Theorem \ref{theorem-2.3} that
$\K^{T'}$ is singular).
\end{remark}

As a corollary of the above constructions and Corollary \ref{corollary-2.4}
we get
\begin{proposition}
\label{proposition-2.6}
Let $\B_u$ be a Springer fiber. Let
$J(u)=\lambda=(\lambda_1\geq \lambda_2\geq\ldots\geq\lambda_k\geq
1).$ If $\lambda_2\geq 2$ and either $k\geq 4$ or $k=3$ and
$\lambda_1\geq 3,\ \lambda_3\geq 2$ then $\B_u$ admits a singular
component.
\end{proposition}
\begin{proof}
Indeed, if $k\geq 4$ then $\lambda_1\geq 2,\lambda_2\geq
2,\lambda_3\geq 1,\lambda_4\geq 1$. Taking any
$T'\in\bT_{\lambda}$ such that $\pi_{1,6}(T')=S$ with $S$ in section \ref{2.5},
we get by Corollary \ref{corollary-2.4} that $\K^{T'}$ is singular.

If $k=3$ then $\lambda_1\geq 3,\ \lambda_2,\lambda_3\geq 2.$
Taking any $T'\in\bT_{\lambda}$ such that $\pi_{1,7}(T')=T$ with $T$ as in
Proposition \ref{proposition-singcomponent-322}, we get by Corollary \ref{corollary-2.4} that $\K^{T'}$ is
singular.  
\end{proof}

\section{Combinatorics of equinonsingular components}

\subsection{Definition and notation}

\label{3.0}

Given any standard tableaux $T,S$ we call the components $\K^T$, $\K^S$
equinonsingular if they are either both singular or both
nonsingular.

For example, let $T$ be a tableau with $n$ in the last column,
then, by Theorem \ref{theorem-2.3}, $\K^T$ and $\K^{\pi_{1,n-1}(T)}$ are
equinonsingular.

In a few following subsections we will construct equinonsingular
components of the same Springer fiber. But before we need to set
combinatorial notation. Let $\lambda$ be a partition of $n$ and
$\lambda^*=( \lambda_1^*,\ldots,\lambda_m^*)$ be its conjugate
partition. For $i,j\ :\  1\leq i\leq j\leq m$ put
$\lambda_{[i,j]}$ to be the partition conjugate to
$(\lambda_i^*,\ldots,\lambda_j^*)$.  For $i\ :\  1\leq i< m$ put
$\varsigma_i:=\sum_{j=1}^i\lambda_j^*$. Note that
$\lambda_{[1,i]}$ is a partition of $\varsigma_i.$

Let $T\in\bT_\lambda$ where
$\lambda^*=(\lambda_1^*,\ldots,\lambda_m^*)$. Put $(T)_{i,j}$ to
be the entry in the $i$-th row and $j$-th column of $T$. Put
$T_j:=((T)_{1,j},\ldots,(T)_{\lambda_j^*,j})$ to be the $j$-th
column of $T$. We will write
 $T=(T_1,\ldots,T_m)$. For $i,j\ :\
1\leq i\leq j\leq m$ put $T_{[i,j]}:=(T_i,\ldots,T_j)$ to be the
subtableau consisting of columns $i,\ldots,j$. We will also write
$S=(P_1,P_2,\ldots)$  when $S$ is a concatenation of subtableaux $
P_1,P_2,\ldots$

Given a tableau $T$ of shape $\lambda$ with consecutive entries
$i+1,\ldots,n+i$ we put $St(T)\in\bT_\lambda$ to be its
standardization, that is $(St(T))_{q,r}=(T)_{q,r}-i.$

\subsection{The procedure $T\mapsto C(T)$}

\label{3.1}

We start with a sort of cyclic procedure using jeu de taquin. Let
$\lambda=(\lambda_1,\ldots,\lambda_k)$ be such that
$\lambda_1=\ldots=\lambda_j$ and ($\lambda_{j+1}<\lambda_1$ or
$j=k$). Let $T\in \bT_\lambda$ be such that
$\sh(\pi_{2,n}(T))=(\lambda_1,\ldots,\lambda_j-1,\ldots)$. Let
$S\in\bT_\lambda$ be obtained from $St(\pi_{2,n}(T))$ by adding
a box of entry $n$ at the end of the $j$-th row. Set $C(T)=S$ and respectively
$C^{-1}(S)=T$.

For example, take
$$T=\begin{array}{lll}
1&2&4\\
3&6&8\\
5&7&10\\
9&11&\\
\end{array}.$$
Then
$$ \pi_{2,11}(T)=\begin{array}{lll}
2&4&8\\
3&6&10\\
5&7&\\
9&11&\\
\end{array},\quad St(\pi_{2,11}(T))=\begin{array}{lll}
1&3&7\\
2&5&9\\
4&6&\\
8&10&\\
\end{array},\quad C(T)=\begin{array}{lll}
1&3&7\\
2&5&9\\
4&6&11\\
8&10&\\
\end{array}.$$

As a straightforward corollary of Theorem \ref{theorem-2.3}
we get
\begin{lemma}
\label{lemma-3.1}
The components ${\mathcal K}^T$ and ${\mathcal K}^{C(T)}$ are
equinonsingular.
\end{lemma}
\begin{proof}
By section \ref{2.2} and Theorem \ref{theorem-2.3},
$\K^{\pi_{2,n}(T)}$ is equinonsingular to $\K^T$ and $\K^{C(T)}$
is equinonsingular to $\K^{\pi_{2,n}(T)}$ so that $\K^T$ and
$\K^{C(T)}$ are equinonsingular.  
\end{proof}
\begin{remark}
According to section \ref{2.2} and the proof of Theorem
\ref{theorem-2.3}, the components $\K^T$ and $\K^{C(T)}$ are both
fiber bundles over the same space with the same fiber isomorphic
to $\K^{\pi_{2,n}(T)}$, so that $\K^T$ and $\K^{C(T)}$ are indeed
equisingular and also have the same Poincar\'e polynomial. We
prefer the notion of ``equinonsingularity'' to ``equisingularity''
since all nonsingular components are nonsingular in the same way
so that we can speak about equinonsingularity of $\K^T$ and
$\K^{\pi_{i,j}(T)}.$
\end{remark}

\subsection{Sch\"utzenberger involution}

\label{3.2}

Let us recall in short Sch\"utzenberger procedure $T\mapsto
Sch(T)$ (cf. \cite{Ful} and \ref{2.2}). Given $T$ let
$D_i=\sh(\pi_{n+1-i,n}(T))$. In such a way we get a correspondence between $T$ and the 
chain of Young diagrams $(D_1,\ldots,D_n)$.  Then $Sch(T)$ is the
tableau obtained from the chain $(D_1,\ldots,D_n)$ by putting $i$
in the only new box of $D_i$ compared to $D_{i-1}$.

For example,
$$T=\begin{array}{lll}
1&2&3\\
4&5&\\
6&&\\
\end{array}\!\!\rightarrow((1),\, (1,1),\, (2,1),\, (2,1,1),\, (2,2,1),\, (3,2,1)),\ Sch(T)=
\begin{array}{lll}
1&3&6\\
2&5&\\
4&&\\
\end{array}\!\!\!.$$

By section \ref{2.2} we have $\K^T\cong \K^{Sch(T)}$. In
particular $\K^T$ and $\K^{Sch(T)}$ are equinonsingular.

\begin{remark}
\label{remark-3.3} Note that the procedure defined in
\ref{3.1} is ``a first step'' of Sch\"utzenberger procedure. We
can in general define $T\rar C(T)$ by adding $n$ to the
standardization of $\pi_{2,n}(T)$ in the empty box obtained by jeu
de taquin, even if this box does not lie in the last column.
However this procedure does not automatically preserve the
singularity. As an example of non preserving singularity by the
general procedure $S\rar T$ let us consider $\K^S$ from section
\ref{2.5}. It is the only singular component in type $(2,2,1,1).$
One has
$$ S=\begin{array}{ll}
1&3\\
2&5\\
4&\\
6&\\
\end{array},\quad \pi_{2,6}(S)=\begin{array}{ll}
2&3\\
4&5\\
6&\\
\end{array},\quad {\rm so\ that}\quad S\rar T=\begin{array}{ll}
1&2\\
3&4\\
5&\\
6&\\
\end{array}$$
however $\K^T$ is nonsingular.
\end{remark}

\subsection{Partitioning a tableau and partial procedures}

\label{3.4}

Let us further note that in special cases we can partition a
tableau $T$ into subtableaux and apply our procedure to the
subtableaux.  Let $\lambda^*=(\lambda_1^*,\ldots,\lambda_m^*)$.
Recall that $\varsigma_i=\lambda_1^*+\ldots+\lambda_i^*$. Consider
$T\in\bT_\lambda$. Note that $(T)_{1,i+1}\leq  \varsigma_i+1$
always and moreover $(T)_{1,i+1}= \varsigma_i+1$ if and only if
$(T_1,\ldots,T_i)\in\bT_{\lambda_{[1,i]}}.$ Thus, if there exists
$i: 1\leq i<m$ such that $(T)_{1,i+1}=\varsigma_i+1$  we can
partition $T$ into $T_{[1,i]}$ and $T_{[i+1,m]}$ where
$St(T_{[i+1,m]})\in \bT_{\lambda_{[i+1,m]}}$   is well-defined.

\begin{proposition}
\label{proposition-3.4}
Let $\lambda^*=(\lambda_1^*,\ldots,\lambda_m^*)$ and let
$T\in\bT_\lambda$ be such that there exists $i\ :\ 1\leq i<m$ such
that $(T)_{1,i+1}=\varsigma_i+1.$ Then
$$\K^T\cong \K^{T_{[1,i]}}\times \K^{St(T_{[i+1,m]})}.$$
\end{proposition}
\begin{proof}
To prove the proposition we recall the following very simple fact:

Let $X,Y$ be algebraic varieties. Let $A\subset X$ and $B\subset
Y$. Then for the Zariski closures:
$$\overline{A\times B}=\overline{A}\times \overline{B}.\eqno{(2)}$$

Now, let
$W_1=\mathrm{ker}\,u^i$ and $W_2=V/W_1$. Let
$u_1\in\mathrm{End}(W_1)$ and $u_2\in\mathrm{End}(W_2)$ be the
nilpotent endomorphisms induced by $u$. Let ${\mathcal B}_{u_1}$
and ${\mathcal B}_{u_2}$ be the corresponding Springer fibers. The
map
\begin{eqnarray}
\Phi: {\mathcal B}_{u_1}\times {\mathcal B}_{u_2} & \rightarrow & {\mathcal B}_u \nonumber \\
(V_0,\ldots,V_{\varsigma_i}),(V'_0,\ldots,V'_{n-{\varsigma_i}}) & \mapsto &
(V_0,\ldots,V_{\varsigma_i},V'_1+W_1,\ldots,V'_{n-\varsigma_i}+W_1)
\nonumber
\end{eqnarray}
is well-defined and is a closed immersion, of image
$\{(V_0,\ldots,V_n)\in{\mathcal B}_u:V_{\varsigma_i}=\ker\,u^i\}$.
We have $\F_{T_{[1,i]}}\times \F_{St(T_{[i+1,m]})}\subset
\Phi^{-1}(\F_T)\subset \Phi^{-1}({ \K}^T)$. By $(2)$ we get ${
\K}^{T_{[1,i]}}\times { \K}^{St(T_{[i+1,m]})}\subset \Phi^{-1}({
\K}^T)$. Thus, $\Phi({ \K}^{T_{[1,i]}}\times {
\K}^{St(T_{[i+1,m]})})\subset { \K}^T$.

Since $\dim\,{ \K}^T=\dim\,{\K}^{T_{[1,i]}}+\dim\,{
\K}^{St(T_{[i+1,m]})}=\dim\,\Phi({ \K}^{T_{[1,i]}}\times
{\K}^{St(T_{[i+1,m]})})$, we obtain the equality $\Phi({
\K}^{T_{[1,i]}}\times { \K}^{St(T_{[i+1,m]})})= { \K}^T$. Hence ${
\K}^{T_{[1,i]}}\times { \K}^{St(T_{[i+1,m]})}\cong { \K}^T$.  
\end{proof}

\label{3.5}

Let $T\in\bT_\lambda$ be such  that there exist $i_0=0<i_1<\ldots<
i_s=m$ where $s\geq 2$ such that $(T)_{1,i_j+1}=\varsigma_{i_j}+1$
for any $j\ :\ 1\leq j<s$. Then we can partition $T$ into
subtableaux $T_{[i_{j-1}+1,i_j]}$ for $1\leq j\leq s$ and try to
apply either the cyclic or the Sch\"utzenberger  procedure to
these subtableaux.

Indeed if $T_{[i_{j-1}+1,i_j]}$ is  such that
$S:=C(St(T_{[i_{j-1}+1,i_j]}))$ is defined, then let $\widehat S$
be obtained from $S$ by $(\widehat
S)_{q,r}=(S)_{q,r}+\varsigma_{i_{j-1}}.$ Then by Lemma \ref{lemma-3.1}
and Proposition \ref{proposition-3.4}, $\K^T$ is equinonsingular to $\K^P$,
where $P$ is the concatenation
$$P=(T_{[1,i_{j-1}]},\widehat S,T_{[i_j+1,m]}).$$
We put $C_{[i_{j-1}+1,i_j]}(T):=P$ in this case. Note that $C(T)=C_{[1,m]}(T).$

For example, let
$$T=\begin{array}{lll}
1&2&5\\
3&4&6\\
\end{array}.$$
Then $T_{[1,2]}$ satisfies the condition. One has
$$ C(T_{[1,2]})=\begin{array}{ll}
1&3\\
2&4\\
\end{array}
\quad{\rm and}\quad C_{[1,2]}(T)=\begin{array}{lll}
1&3&5\\
2&4&6\\
\end{array}.
$$
Note that in general $C_{[i_{j-1}+1,i_j]}(T)\ne C^k(T)$ for any
$k\in{\mathbb Z}.$ Indeed, in the case above
$$S:=C(T)=\begin{array}{lll}
1&3&4\\
2&5&6\\
\end{array},\quad U:=C(S)=\begin{array}{lll}
1&2&3\\
4&5&6\\
\end{array},\quad T=C(U)=C^3(T)
$$
so that $C_{[1,2]}(T)$ cannot be obtained from $T$ by applying our
cyclic procedure a few times to any of them.

\smallskip
Exactly in the same way let $S:=Sch(St(T_{[i_{j-1}+1,i_j]}))$ and
let $\widehat S$ be obtained from $S$ by $(\widehat
S)_{q,r}=(S)_{q,r}+\varsigma_{i_{j-1}}.$ Then by \ref{3.2} and
Proposition \ref{proposition-3.4}, $\K^T$ and $\K^P$ are equinonsingular,
where $P=(T_{[1,i_{j-1}]},\widehat S,T_{[i_j+1,m]})$. We put
$Sch_{[i_{j-1}+1,i_j]}(T):=P$ in this case. Again,
$Sch(T)=Sch_{[1,m]}(T).$

\subsection{Definition of $Eqs$-classes of equinonsingular components}

\label{3.6}

Using the procedures above let us define classes of
equinonsingular components in $\B_u$ and respectively classes of
equinonsingular tableaux in $\bT_{J(u)}$.

Given $T\in\bT_{\lambda}$ set $Eqs(T)$ to be the set of all
$S\in\bT_\lambda$ for which there exists a chain
$T=T^1,\ldots,T^k=S$ where $T^l$ is obtained from $T^{l-1}$ in one
of two  ways
\begin{itemize}
\item[(a)]  $T^l=C_{[i,j]}^k(T^{l-1})$ where $1\leq i<j\leq m$ and $k=\pm 1;$
\item[(b)]  $T^l=Sch_{[i,j]}(T^{l-1})$ where $1\leq i<j\leq m$.
\end{itemize}

\subsection{On the two-row case}

\label{3.7}

As a first application we give a simple proof of the fact shown in
\cite[\S 5]{F}, namely
\begin{theorem}
\label{theorem-3.7} If $J(u)=(r,s)$ then all the components of
$\B_u$ are nonsingular. Moreover, they are iterated bundles of
base type $(\Co{\mathbb P}^1,\ldots, \Co{\mathbb P}^1)$ with $s$
terms.
\end{theorem}
\begin{proof}
Set
$$P(r,s):=\begin{array}{lllcll}1&3&\cdots&2s-1&\cdots&r+s\\ 2&4&\cdots&2s&&\\ \end{array}.\eqno{(3)}$$
By Proposition \ref{proposition-3.4}
$$\begin{array}{rcccc}
\K^{P(r,s)}&\cong& \underbrace{\B_2\times\B_2\times\ldots\times\B_2}&\times & \underbrace{\B_1\ldots\times\B_1}.\\
&& \mbox{\footnotesize${ s\ \ {\rm times}}$}& &\mbox{\footnotesize
$r-s\ \ {\rm times}$}
\\ \end{array}$$
Since all the components in a same $Eqs$-class are either
isomorphic or at least iterated bundles of the same base type
(see the remark in \ref{3.1}), it remains to show that
for $T\in\bT_{(r,s)}$ one has $P(r,s)\in Eqs(T).$ We show this by
induction on $r+s$. For $r+s\leq 2$ this is trivially true. Assume
this is true for $r+s\leq n-1$ and show for $r+s=n$.

First consider the case $r=s$. If $(T)_{2,1}=2$ then $T$ is the
concatenation $T=({}^1_2,\pi_{3,n}(T))$ with
$St(\pi_{3,n}(T))\in\bT_{(r-1,r-1)}$ so that all the tableaux with
$(T)_{2,1}=2$ are in the same class by induction hypothesis. Thus,
it is enough to show that for any $T\in\bT_{(r,r)}$ there exists
$j$ such that $(C^j(T))_{2,1}=2.$ Actually, if $(T)_{2,1}=j+1$
where $j\geq 2$ then $(C(T))_{2,1}=j$, so that
$(C^{j-1}(T))_{2,1}=2$, which completes the proof in the case
$(r,r)$.

Now assume $s<r$. If $\sh (\pi_{2,n}(T))=(r-1,s)$ then
$C(T)=(St(\pi_{2,n}(T)),(n))$ and $C(T)_{[1,r-1]}\in\bT_{(r-1,s)}$
so that by induction hypothesis $P(r-1,s)\in Eqs(C(T)_{[1,r-1]})$.
Thus, $P(r,s)=(P(r-1,s),(n))\in Eqs(C(T))=Eqs(T)$.

If $\sh(\pi_{2,n}(T))=(r,s-1)$ then there exists $i\ :\ 1\leq
i\leq s$ such that $(T)_{2,i}<(T)_{1,i+1}$, hence one has
$\sh(\pi_{1,(T)_{2,i}}(T))=(i,i)$ so that $(T)_{2,i}=2i$. Then
$T=(T_{[1,i]},T_{[i+1,r]})$ with $T_{[1,i]}\in \bT_{(i,i)}$ and
$St(T_{[i+1,r]})\in \bT_{(r-i,s-i)}$. By induction hypothesis
$P(i,i)\in Eqs(T_{[1,i]})$ and $P(r-i,s-i)\in
Eqs(St(T_{[i+1,r]}))$. Thus, again $P(r,s)\in Eqs(T)$.  
\end{proof}

\section{On the components of $(r,s,1)$ type}

\subsection{Outline}

\label{4.0}

Let us now consider the case $(r,s,1)$.
Again applying Theorem \ref{theorem-2.3}, it is enough to show that $\K^T$
is nonsingular for $T\in\bT_{(r,r,1)}.$

Set
$$Q(k,k,1)=\begin{array}{cccc}
1&3&\ldots&k+1\\
2&k+3&\ldots&2k+1\\
k+2&&\\
\end{array}.$$
Let $P(s,s)$ be  defined as in formula $(3)$ . Put $P(0,0):=\emptyset$, and let $P(s,s|t)$ be the
tableau $P(s,s)$ shifted by $t$, that is
$(P(s,s|t))_{i,j}=(P(s,s))_{i,j}+t$ for $1\leq i\leq 2$ and $1\leq
j\leq s.$ Recall that $(Q(k,k,1),P(r-k,r-k|2k+1))\in\bT_{(r,r,1)}$
is the tableau obtained by concatenating $Q(k,k,1)$ and
$P(r-k,r-k|2k+1)$. Using Proposition \ref{proposition-3.4}, we get
$$
\begin{array}{rccc}
\K^{(Q(k,k,1),P(r-k,r-k|2k+1))}&\cong&\K^{Q(k,k,1)}\times& \underbrace{\B_2\times\ldots\times\B_2}.\\
&&& {\mbox{\footnotesize $r-k\ \ {\rm times}$}} \\
\end{array}
$$

Our strategy is to show first that $\bT_{(r,r,1)}$ is partitioned
into $r$ classes of equinonsingular components
$\{Eqs((Q(k,k,1),P(r-k,r-k|2k+1))\}_{k=1}^r$ and then to show that
the component ${\mathcal K}^{Q(k,k,1)}$ is nonsingular.

\subsection{Partition into $Eqs$-classes in the $(r,r,1)$ case}

\label{4.1}

We start with some combinatorial notes. Consider
$T\in\bT_\lambda$ for  a partition $\lambda$ of $n$.  For $1\leq
i\leq n$ put $r_T(i)$ to be the number of the row $i$ belongs to,
that is $r_T(i)=k$ if there exists $l$ such that $(T)_{k,l}=i.$
Since the entries increase in the rows from left to right, the
positions $\{r_T(i)\}_{i=1}^n$ define $T$ completely.

Put $\tau(T):=\{i\ :\ r_T(i+1)>r_T(i)\}.$ Note that if $i\in
\tau(T)$ then for any $k,l$ such that $k\leq i$ and $l\geq i+1$
one has $i\in\tau(\pi_{k,l}(T)).$  In fact
$\tau(\pi_{k,l}(T))=\tau(T)\cap\{k,k+1,\ldots,l-1\}$.

One also has  $i\in\tau(T)$ if and only if
$\pi_{i,i+1}(T)={}_{i+1}^{\ i}.$ Since
$$\sh(\pi_{i,j}(T))=\sh(\pi_{n+1-j,n+1-i}(Sch(T)))$$ we get that
$i\in\tau(T)$ if and only if $n-i\in\tau(Sch(T)).$

Let $T\in\bT_\lambda$ where $\lambda=(r,s,1)$. Then $(T)_{2,1}-1,
(T)_{3,1}-1\in\tau(T).$ Put
$$j(T):=\max\{i\in \tau(T)\ :\
i<(T)_{3,1}-1\}.$$ Since $j(T)\geq (T)_{2,1}-1$, it is well-defined. Put
$${\rm dist}(T):=(T)_{3,1}-1-j(T).$$

We need the following technical
\begin{lemma}
\label{lemma-4.1} Let $T\in\bT_{(r,s,1)}$. Then
\begin{itemize}
\item[\rm (i)] One has $(Sch(T))_{3,1}=n-j(T)+1$ so that ${\rm
dist}(Sch(T))={\rm dist}(T);$
\item[\rm (ii)] If $C(T)$ is
defined then $(C(T))_{3,1}=(T)_{3,1}-1$ and $j(C(T))=j(T)-1$ (for $r>1$),
so that ${\rm dist}(C(T))={\rm dist(T)};$
\item[\rm (iii)] For any $S\in Eqs(T)$ one has ${\rm
dist}(S)={\rm dist}(T).$
\end{itemize}
\end{lemma}
\begin{proof}
(i)  Let $T\in\bT_{(r,s,1)}$. We show $(Sch(T))_{3,1}=n-j(T)+1$ by
induction on $n=r+s+1.$ This is trivially true if $n=3$. Assume
this is true for $n-1$ and show for $n$.
\begin{itemize}
\item[(a)] If $(T)_{3,1}=n$ then $j(T)=\max\{i\in\tau(T),i<n-1\}.$
Note that  $(\pi_{j(T),n}(T))_{2,1}=j(T)+1$ and
$(\pi_{j(T),n}(T))_{3,1}=n$. Note also that
$$\tau(\pi_{j(T),n}(T))=\{j(T),n-1\}.\eqno{(4)}$$
One has
 $$\pi_{j(T),n}(T)=\begin{array}{ll}
 j(T)&\cdots\\
 j(T)+1&\cdots\\
 n&\\
 \end{array}.$$
If $\sh(\pi_{j(T),n}(T))=(r',s',1)$ where $s'>1$ then $(
\pi_{j(T),n}(T))_{2,2}$ exists. But then by properties of Young
tableaux  $(\pi_{j(T),n}(T))_{2,2}> j(T)+2$ so that
$(\pi_{j(T),n}(T))_{2,2}-1\in \tau(\pi_{j(T),n}(T))$ in
contradiction to $(4)$. Therefore one has
$\sh(\pi_{j(T),n}(T))=(r',1,1)$ and $\sh(\pi_{j(T)+1,n}(T))=(r',1)$ so that, by Sch\"utzenberger
procedure, $(Sch(T))_{3,1}=n-j(T)+1.$
\item[(b)]
Assume that $(T)_{3,1}<n.$ Then by induction hypothesis
$(Sch(\pi_{1,n-1}(T)))_{3,1}=n-j(T).$ Note that up to
standardization $Sch(\pi_{1,n-1}(T))=\pi_{2,n}(Sch(T))$. In
particular we get $\mathrm{sh}\,(\pi_{2,n}(Sch(T)))=(r',s',1)$ so that jeu de
taquin from $Sch(T)$ to $\pi_{2,n}(Sch(T))$ does not touch the
third row, that is $(Sch(T))_{3,1}=(\pi_{2,n}(Sch(T)))_{3,1}$.
Hence $(Sch(T))_{3,1}=(Sch(\pi_{1,n-1}(T)))_{3,1}+1=n-j(T)+1.$
\end{itemize}
Since there are no elements of $\tau$ between $j(T)$ and
$(T)_{3,1}-1$ we get that $n-(T)_{3,1}+1$ is the largest element
of $\tau(Sch(T))$ smaller than $n-j(T)$, so that
$j(Sch(T))=n-(T)_{3,1}+1$ and ${\rm
dist}(Sch(T))=n-j(T)-(n-(T)_{3,1}+1)=(T)_{3,1}-1-j(T)={\rm
dist}(T).$

\bigskip
(ii) is immediate for $r=1$. Assume $r>1$. If $C(T)$ is defined
then $(\pi_{2,n}(T))_{3,1}=(T)_{3,1}$ so that
$(C(T))_{3,1}=(T)_{3,1}-1$. Also, in this case either
$(T)_{2,1}>2$ (this is always the case if $s<r$), or $(T)_{2,1}=2$
and $4\leq(T)_{2,2}<(T)_{3,1}$ (then $(T)_{2,2}-1\in\tau(T)$), so that in
any case $j(T)>1$ and thus $j(C(T))=j(T)-1$, which provides the
result.

\bigskip
(iii) follows  straightforwardly from (i) and (ii)
by definition of $Eqs(T)$.  
\end{proof}

Now, we can prove the following

\begin{proposition}
We have
$\bT_{(r,r,1)}=\coprod\limits_{k=1}^r Eqs(Q(k,k,1),P(r-k,r-k|2k+1))$.
\end{proposition}
\begin{proof}
Note that ${\rm dist}(Q(k,k,1),P(r-k,r-k|2k+1))=k$. Thus, by Lemma
\ref{lemma-4.1}, for any $T\in Eqs(Q(k,k,1),P(r-k,r-k|2k+1))$ one
has ${\rm dist}(T)=k$. It remains to show that ${\rm dist}(T)=k$
implies
$$(Q(k,k,1),P(r-k,r-k|2k+1))\in Eqs(T).$$

This is trivially true for $r=1.$ Assume this is true for $r-1$
and show for $r.$ First note that  for  any $T\in\bT_{(r,r,1)}$
one has either $(T)_{2,1}=2$ or $(C^{(T)_{2,1}-2}(T))_{2,1}=2$
exactly as in the case $(r,r).$ Thus, there exists $S\in Eqs(T)$
such that $(S)_{2,1}=2.$
\begin{itemize}
\item[(a)] If $(S)_{2,2}<(S)_{3,1}$ then $U:=C^2(S)$ is defined.
Note that $(U)_{1,r}=n-1$ and $(U)_{2,r}=n$. Indeed
$$C(S)=\begin{array}{ccccc}
1&(S)_{1,2}-1&\ldots&(S)_{1,r-1}-1&(S)_{1,r}-1\\
(S)_{2,2}-1&(S)_{2,3}-1&\ldots&(S)_{2,r}-1&n\\
(S)_{3,1}-1&&&&\\
\end{array}$$
and respectively, since
$(C(S))_{1,i}=(S)_{1,i}-1<(C(S))_{2,i-1}=(S)_{2,i}-1$ one has
$$ U=C^2(S)= \begin{array}{ccccc}
(S)_{1,2}-2&(S)_{1,3}-2&\ldots&(S)_{1,r}-2&n-1\\
(S)_{2,2}-2&(S)_{2,3}-2&\ldots&(S)_{2,r}-2&n\\
(S)_{3,1}-2&&&&\\
\end{array}.$$
Further, by induction, $(Q(k,k,1),P(r-1-k,r-1-k|2k+1))\in
Eqs(\pi_{1,n-2}(U))$ for $k={\rm dist}(U)$, so that one gets
$(Q(k,k,1),P(r-k,r-k|2k+1))\in Eqs(T)$ by the definition of
$Eqs(T)$.

\item[(b)] If $(S)_{2,2}>(S)_{3,1}$ then in particular ${\rm
dist}(S)=(S)_{3,1}-1-1=(S)_{3,1}-2$ and:
\begin{itemize}
\item{} $3\leq (S)_{3,1}\leq r+2$;
\item{} $(S)_{1,i}=i+1$ for any $i\ :\ 2\leq i\leq (S)_{3,1}-2;$
\item{} $(S)_{2,r}=n.$
\end{itemize}

If $(S)_{3,1}=3$ then $S=(Q(1,1,1),\pi_{4,n}(S))$ and, by the
proof of Theorem \ref{theorem-3.7}, we have $P(r-1,r-1)\in
Eqs(St(\pi_{4,n}(S)))$. Thus, $(Q(1,1,1),P(r-1,r-1|3))\in Eqs(T)$.

If $(S)_{3,1}=r+2$ then $S=Q(r,r,1)$.

Thus, it remains to consider the case $3< (S)_{3,1}< r+2.$ Let us
show that in this case there exists $U\in Eqs(S)$ such that
$(U)_{1,r}=n-1,\ (U)_{2,r}=n$. As in (a), this provides by
induction hypothesis that $(Q(k,k,1),P(r-k,r-k|2k+1))\in Eqs(T)$.

Put $m:=(S)_{3,1}.$ Note that $1,m-1\in\tau(S)$. As well, since
$m<r+2$ one has $r+2\leq (S)_{1,r}\leq n-1$, and in particular
$(S)_{1,r}>m-1$ on one hand and $(S)_{1,r}$ is the maximal element of
$\tau(S)$ on the other hand. Put $l:=(S)_{1,r}$. One has
$\tau(S)=\{1,m-1,\ldots, l\}.$ If $l=n-1$ we are done.

If $l<n-1$  consider $S':=Sch(S).$ One has
$\tau(S')=\{n-l,\ldots,n-m+1,n-1\}$. Thus,
$$S'=\begin{array}{llcccccc}
1 &2&\cdots\  n-l   &\cdots&\cdots\cdots&\cdots&n-m+1\\
n-l+1&\cdots&\cdots\cdots    &\cdots&n-m+2& \cdots&  n-1\\
n   &&&&&&\\
\end{array}.$$
Set $U:=C^{n-l+1}(S')$. Let us show that one has $(U)_{1,r}=n-1,\
(U)_{2,r}=n.$ Consider first $V:=C^{n-l-1}(S').$ Since the first
$n-l$ entries of the first row of $S'$ are $1,2,\ldots,n-l$ we get
that $(V)_{2,1}=n-l+1-(n-l-1)=2$ and correspondingly
$(V)_{3,1}=n-(n-l-1)=l+1.$ Also, since $l\geq r+2$ we get that not
more than the last $r-2$ elements of the second row of $V$ can be
$l+2,l+3,\ldots, n$. Thus, $(V)_{2,2}<l+2$. Therefore
$(V)_{2,2}<(V)_{3,1}$ and $C(V)$ is defined. Summarizing, we get
$$V=\begin{array}{cccc}1&x_2&\ldots&x_r\\
2&y_2&\ldots&n\\
l+1&&&\\
\end{array},\quad
C(V)=\begin{array}{ccccc}
1&x_2-1&\ldots&x_{r-1}-1&x_r-1\\
y_2-1&y_3-1&\ldots&n-1&n\\
l&&&&\\
\end{array},$$
$$
U=C^2(V)=\begin{array}{ccccc}
1&x_3-2&\ldots&x_r-2&n-1\\
y_2-2&y_3-2&\ldots&n-2&n\\
l-1&&&&\\
\end{array}.$$
\end{itemize}
The proof is now complete.  
\end{proof}

\subsection{A set of special flags in the component $\K^{Q(k,k,1)}$}

\label{4.3}

It remains to show that $\K^{Q(k,k,1)}$ is nonsingular. To show
this we use the results of \cite[\S 1.3]{Fr} which we formulate in
short in the beginning of this subsection.

Let $\lambda^*=(\lambda_1^*,\ldots,\lambda_k^*)$ be a partition of
$n$ and let $\{e_i\}_{i=1}^n$ be a fixed standard (Jordan) basis
of $V=\Co^n$. The basis is enumerated according to some standard
tableau $T\in\bT_{\lambda}$
$$T=\begin{array}{llll}
1&t_{1,2}&\ldots&t_{1,k}\\
t_{2,1}&\ldots&&\\
\vdots&\vdots&&\\
t_{\lambda_1^*,1}&\cdots&&\\
\end{array}
$$
such that $u\in {\rm End}(V)$ of Jordan form $J(u)=\lambda$ acts
on this basis by
$$u(e_{t_{i,j}})=\left\{\begin{array}{ll}
e_{t_{i,j-1}}&{\rm if}\ j\geq 2;\\
0&{\rm otherwise}.\\
\end{array}\right.$$
For any $\sigma\in \bS_n$  put $F_\sigma=(V_0,\ldots,V_n)$ to be the
flag such that, for any $i\ :\ 1\leq i\leq n$, $V_i={\rm
Span}\{e_{\sigma(1)},\ldots e_{\sigma(i)}\}$. We call such flags
Jordan flags.

Note that $F_\sigma\in\B_u$ if and only if for any $i\ :\ 1\leq
i\leq n$ one has $u(e_{\sigma(i)})\in V_{i-1}.$ Put
$\bS_u:=\{\sigma\in\bS_n\ :\ F_\sigma\in\B_u\}.$

Using \cite[Lemma 1.3]{Fr} one has
\begin{proposition}
\label{proposition-4.3}
\begin{itemize}
\item[\rm (i)] $\K^S\cap\{F_\sigma\}_{\sigma\in\bS_u}\ne\emptyset$ for
any $S\in\bT_{\lambda}.$
\item[\rm (ii)] The component $\K^S$ is nonsingular if and
only if  any $F\in \K^S\cap\{F_\sigma\}_{\sigma\in\bS_u}$
is a nonsingular point of $\K^S.$

\end{itemize}
\end{proposition}

Let us apply Proposition \ref{proposition-4.3} to our situation.
Put $n=2k+1$. Let $\{e_i\}_{i=1}^n$ be a standard (Jordan) basis
of $V=\Co^n$ corresponding to
$$T=\begin{array}{llll}
1&3&\cdots&n-2\\
2&4&\cdots&n-1\\
n&&\\
\end{array}$$
so that for the fixed nilpotent $u\in {\rm End}(V)$ of Jordan form
$(k,k,1)$ one has
$$u(e_i)=\left\{\begin{array}{ll} 0& {\rm if}\ i=1,2,n;\\
e_{i-2}&{\rm otherwise}.\\
\end{array}\right.
$$

Let $F_\sigma$ denote a Jordan flag. Note that $\sigma\in\bS_u$ if
and only if $\sigma^{-1}$ is increasing on $\{1,3,\ldots,n-2\}$
and $\{2,4,\ldots,n-1\}$, that is $\sigma$ is a shuffling of the
sets $\{1,3,\ldots,n-2\},\ \{2,4,\ldots,n-1\}$ and $\{n\}.$

Note that $\ker u={\rm Span}\{e_1,e_2,e_n\}\subset V_{k+2}$ for
any flag $F=(V_0,\ldots, V_n)\in \F_{Q(k,k,1)}$. Since
$\K^{Q(k,k,1)}=\ov{\F_{Q(k,k,1)}}$ we get that $\ker u\subset
V_{k+2}$ for any $F\in \K^{Q(k,k,1)}.$ In particular for
$F_\sigma\in \K^{Q(k,k,1)}$ one has
$$\sigma^{-1}(n)\leq k+2. \eqno{(5)}$$

For $d\in\{3,\ldots,k+2\}$ put
$$(d):=(d,n,n-1,\ldots,d+1)=\left(\begin{array}{ccccccc}
1&\ldots&d-1&d&d+1&\ldots&n\\
1&\ldots&d-1&n&d&\ldots&n-1\\
\end{array}\right)\in \bS_n.$$
We have the following
\begin{lemma}
\label{lemma-4.4}
\begin{itemize}
\item[\rm (i)] $\{F_{(d)}\}_{d=3}^{k+2}\cap \K^{Q(k,k,1)}\ne\emptyset$.
\item[\rm (ii)] The component $\K^{Q(k,k,1)}$ is nonsingular if and
only if every Jordan flag $F_{(d)}$ such that $F_{(d)}\in \K^{Q(k,k,1)}$ is a
nonsingular point of $\K^{Q(k,k,1)}$.
\end{itemize}
\end{lemma}
\begin{proof}
Put $Z(u)=\{g\in GL(V):gug^{-1}=u\}$ to be the stabilizer of $u$.
Define $A_t^{(i)},B_t^{(j)}\in GL(V)$ (for $j\geq 2$) as follows
\begin{eqnarray*}
&A_t^{(i)}(e_s)=\left\{\begin{array}{ll}
e_s&{\rm if}\ s<n;\\
e_n+te_i&{\rm if}\ s=n;\\
\end{array}\right.\\
&B_t^{(j)}(e_s)=\left\{\begin{array}{llll}
e_s&{\rm if}\ s\not\equiv j\mbox{ mod 2} \mbox{ \ or } s\in\{1,n\};\\
e_s+te_{s-1}&{\rm if}\ s\equiv j\mbox{ mod 2} \mbox{ \ and } s\notin\{1,n\}.\\
\end{array}\right.\end{eqnarray*}
Note that  $A_t^{(i)}\in Z(u)$ for $i=1,2$ and $B_t^{(j)}\in Z(u)$
for $j<n.$

By Proposition \ref{proposition-4.3},
$\K^{Q(k,k,1)}\cap\{F_\sigma\}_{\sigma\in\bS_u}\ne\emptyset$, and
$\K^{Q(k,k,1)}$ is nonsingular if and only if any $F_\sigma\in
\K^{Q(k,k,1)}$ is a nonsingular point of it. Note that if $F'\in
\ov{Z(u)(F)}$ is a nonsingular point of $\K^T$ then $F$ is a
nonsingular point of $\K^T.$ Thus, it is sufficient to show that
the closure of the $Z(u)$-orbit of $F_{\sigma}$ for $\sigma\in
\bS_u$ contains some flag $F_{(d)}$. Actually,
\begin{itemize}
\item[(a)] First, note that it is enough to consider only
$F_\sigma\in \K^{Q(k,k,1)}$ for $\sigma$ such that
$\sigma^{-1}(n)>\sigma^{-1}(i)$ for $i=1,2.$ Indeed, if
$F_\sigma\in \K^{Q(k,k,1)}$ is such that
$\sigma^{-1}(n)<\sigma^{-1}(i)$ for $i\in\{1,2\}$ one has
$\sigma':=(i,n)\sigma$   (satisfying
${\sigma'}^{-1}(n)>{\sigma'}^{-1}(i)$) which is obtained as
$F_{\sigma'}=\lim\limits_{t\rightarrow\infty}A_t^{(i)}(F_\sigma)$,
thus, $F_{\sigma'}$ belongs to $\ov{Z(u)(F_\sigma)}$.

\item[(b)] Next, exactly in the same way it is enough to consider
only $F_\sigma$ for $\sigma$ such that
$\sigma^{-1}(i-1)<\sigma^{-1}(i)$ for every $1<i<n$. Indeed, for
$F_\sigma\in \K^{Q(k,k,1)}$  such that
$\sigma^{-1}(i-1)>\sigma^{-1}(i)$ for some $i: 2\leq i\leq n-1$,
there is $\sigma'\in\bS_u$ satisfying
${\sigma'}^{-1}(i-1)<{\sigma'}^{-1}(i)$ which is obtained as
$F_{\sigma'}=\lim\limits_{t\rar\infty}B_t^{(i)}(F_\sigma)$ thus,
$F_{\sigma'}$ belongs to $\ov{Z(u)(F_\sigma)}.$
\end{itemize}
Note that $\sigma$ satisfying these two properties is equal to $(d)$ for
$d=\sigma^{-1}(n)$. One has  $d\leq k+2$ by $(5)$ and $d\geq 3$ by (a).  
\end{proof}

Let $B:={\rm Stab}^t(F_{(d)})=\{A\in GL(V):
A^t(F_{(d)})=F_{(d)}\}$ be the transposition of the Borel subgroup
fixing the flag $F_{(d)}$. Then $\Omega_{(d)}:=B F_{(d)}$ is an
open subset in the flag variety $\B.$  For
$F=(V_0,\ldots,V_n)\in\Omega_{(d)}$ there is a unique basis
$\eta_1,\ldots,\eta_n$ such that
$V_i=\mathrm{Span}\{\eta_{1},\ldots,\eta_{i}\}$ where
$$\eta_{i}=e_{(d)(i)}+\sum_{j=i+1}^n\phi_{i,j}e_{(d)(j)}$$ for some
$\phi_{i,j}\in\Co$.
The maps $F\mapsto \phi_{i,j}$ are algebraic and the map $F\mapsto
(\phi_{i,j})_{1\leq i< j\leq n}$ is an isomorphism from
$\Omega_{(d)}$ to the affine space $\Co^{\frac{n(n-1)}{2}}$.

\medskip

To show the nonsingularity of the component $\K^{Q(k,k,1)}$ we will construct
for every $d\in\{3,\ldots,k+2\}$ a
closed immersion $\Phi:\Co^{k+2}\rightarrow \Omega_{(d)}$
satisfying the two conditions:
\begin{itemize}
\item[(A)] For $(a_1,\ldots,a_{k+2})$ such that $a_i\ne 0$ for any
$i$ one has $\Phi(a_1,\ldots,a_{k+2})\in \F_{Q(k,k,1)}$;
\item[(B)] $\Phi(0,\ldots,0)=F_{(d)}$.
\end{itemize}
By (A) we get that $\Phi:\Co^{k+2}\rar \K^{Q(k,k,1)} $ is an
isomorphism on a locally closed subset of $\K^{Q(k,k,1)}$. Since
$\dim \K^{Q(k,k,1)}=k+2$, the image of $\Phi$ is an open subset of $\K^{Q(k,k,1)}$.
Moreover, being isomorphic to $\Co^{k+2}$, it is  nonsingular. By
(B) the flag $F_{(d)}$ belongs to it, hence $F_{(d)}$ is a nonsingular
element of $\K^{Q(k,k,1)}$. By Lemma \ref{lemma-4.4}  this provides that
$\K^{Q(k,k,1)}$ is nonsingular.

\subsection{A preliminary construction before proving the nonsingularity of $\K^{Q(k,k,1)}$}

\label{4.6}

To construct $\Phi$ we need some preliminary
construction of ``right'' vectors. Let $(e_1,\ldots,e_n)$ be as in
\ref{4.3}. Let $W_1={\rm Span}\{e_1,e_2,\ldots,e_{n-3}\}$ and
$W_2={\rm Span}\{e_3,e_4,\ldots,e_{n-1}\}.$ Thus, $u:W_2\rar W_1$
is an isomorphism. Let $w:W_1\rar W_2$ be its inverse, that is
$w(e_i)=e_{i+2}$ for any $i\ :\ 1\leq i\leq n-3$. We extend the
action of $w$ to ${\rm Span}\{e_1,\ldots,e_{n-1}\}$ putting
$w(e_{n-2})=w(e_{n-1})=0.$

Given $\alpha_3,\ldots,\alpha_{k+1}\in\Co$ we construct
$\{v_i\}_{i=1}^{k+1}\in V$ as follows. Let $v_1=e_1$ and
$v_2=e_2.$ For $i=3,\ldots,k+1$ set $v_i$ inductively by
$$v_i:=w(v_{i-2})+\alpha_iw(v_{i-1}).$$
One has
\begin{lemma}
\label{lemma-4.6}
\begin{itemize}
\item[\rm (i)] For $i\ :\ 3\leq i\leq k+1$ one has
$u(v_i)=v_{i-2}+\alpha_i v_{i-1}$. In particular,
 $V_i:={\rm Span}\{v_1,\ldots,v_i\}$ is $u$-stable
for every $i\ :\ 1\leq i\leq k+1$.
\item[\rm (ii)] For $i\ :\ 2\leq
i\leq k+1$ one has   $v_{i}\in \ker\,u^{i-1}$. Moreover, if
$\alpha_j\ne 0$ for every $j\ :\ 3\leq j\leq k+1$, then $v_i\not\in
\ker u^{i-2}.$
\item[\rm (iii)] Set
$\widetilde{\alpha}_1=\widetilde{\alpha}_2=0$. For $i\geq 3$, set
$\widetilde{\alpha}_i=\widetilde{\alpha}_{i-2}+\alpha_i$. Then
$$v_i-e_i-\widetilde{\alpha}_ie_{i+1}\in {\rm Span}\{
e_{i+2},\ldots,e_{n-1}\}$$ for every $i\ :\ 1\leq i\leq k+1$.
\end{itemize}
\end{lemma}
\begin{proof}
(i) follows from the definition.

We prove (ii) by induction. It is immediate for $i=2$. Case $i=3$:
by (i) $u(v_3)=e_1+\alpha_3e_2\in \ker u$ that is $v_3\in\ker u^2$
and $v_3\not\in \ker u.$ Assume this is true  for $i\geq 3$ and
show for $i+1.$ By (i), $u(v_{i+1})=v_{i-1}+\alpha_{i+1}v_i$. By
induction hypothesis $v_{i-1},v_i\in\ker u^{i-1}$ and if
$\alpha_{j}\ne 0$ for $j\leq i+1$ one has $
v_{i-1}+\alpha_{i+1}v_i\not\in\ker u^{i-2}$. Thus, $v_{i+1}\in \ker
u^i$ and $v_{i+1}\not\in\ker u^{i-1}.$

We prove (iii) by induction. This is trivially true for $i=1,2$.
Assume it is true for $j\leq i$ and show for $i+1.$ By definition
$v_{i+1}=w(v_{i-1})+\alpha_{i+1}w(v_i)=w(v_{i-1})+(\widetilde{\alpha}_{i+1}-\widetilde{\alpha}_{i-1})w(v_i)$.
Further by induction
$v_{i-1}=e_{i-1}+\widetilde{\alpha}_{i-1}e_i+\sum\limits_{j=i+1}^{n-1}\beta_j
e_j$ and
$v_i=e_i+\widetilde{\alpha}_ie_{i+1}+\sum\limits_{j=i+2}^{n-1}\gamma_j
e_j.$
Thus,
$$\begin{array}{rl}
v_{i+1}=&e_{i+1}+\widetilde{\alpha}_{i-1}e_{i+2}+\sum\limits_{j=i+1}^{n-3}\beta_j
e_{j+2}+(\widetilde{\alpha}_{i+1}-\widetilde{\alpha}_{i-1})(e_{i+2}+\widetilde{\alpha}_ie_{i+3}+
\sum\limits_{j=i+2}^{n-3}\gamma_j e_{j+2})\\
=&e_{i+1}+\widetilde{\alpha}_{i+1}e_{i+2}+\sum\limits_{j=i+3}^{n-1}\lambda_je_j.\\
\end{array}$$
Therefore, $v_{i+1}-e_{i+1}-\widetilde{\alpha}_{i+1}e_{i+2}\in
{\rm Span}\{e_{i+3},\ldots,e_{n-1}\}.$  
\end{proof}

To define vectors $v_{k+2},\ldots, v_{n-1}$ we need some
intermediate construction of vectors $r_1^{(i)},\ldots,r_i^{(i)}$
for $i\geq k+2$. Then set $v_i:=r_i^{(i)}$ for $i\ :\ k+2\leq
i\leq n-1.$

First put $r_j^{(k+1)}:=v_j$ and $\beta_j^{(k+1)}:=\alpha_j$ for
$j\ :\ 1\leq j\leq k+1.$

For $i\geq k+2$ we define vectors $r_j^{(i)}$ and numbers
$\beta_j^{(i)}$ for $j\ :\ 1\leq j\leq i$ by induction procedure.
Set $r_1^{(i)}:=e_1,\ r_2^{(i)}:=e_2$ and
$\beta_1^{(i)}=\beta_2^{(i)}=0$ for any $k+2\leq i\leq n-1.$ For
$j\ :\ 3\leq j\leq i$ set $r_j^{(i)}:=w(r_{j-2}^{(i-1)})$ and
$\beta_j^{(i)}:=\beta_{j-2}^{(i-1)}$. We have the following
\begin{lemma}
\label{lemma-4.7}
Let $i\in\{k+1,\ldots,n-1\}$.
\begin{itemize}
\item[\rm (i)] Note that
$r^{(i)}_j=w(r^{(i)}_{j-2})+\beta_j^{(i)}w(r^{(i)}_{j-1})$ and
$u(r^{(i)}_j)=r^{(i)}_{j-2}+\beta_j^{(i)}r^{(i)}_{j-1}$ for any
$1\leq j\leq i$. Therefore, the subspace ${\rm
Span}\{r^{(i)}_1,\ldots,r^{(i)}_{i}\}$ is $u$-stable. Moreover for
$i\geq k+2$ we have ${\rm
Span}\{r^{(i-1)}_1,\ldots,r^{(i-1)}_{i-1}\}\subset {\rm
Span}\{r^{(i)}_1,\ldots,r^{(i)}_{i}\}$.
\item[\rm (ii)] Note that
$\beta^{(i)}_j=0$ and $r^{(i)}_j=e_j$ for $1\leq j\leq 2(i-k)$.
Also $\beta^{(i)}_j=\alpha_{j-2(i-k-1)}$
for $j>2(i-k)$.
Moreover for $j>2(i-k)$, one has $r_j^{(i)}\in\ker
\,u^{j-(i-k)}$, and if
$\alpha_3,\ldots,\alpha_{k+1}$ are all nonzero, then
$r_j^{(i)}\not\in\ker\,u^{j-(i-k)-1}$.
\item[\rm (iii)] Set
$\widetilde{\beta}^{(i)}_1=\widetilde{\beta}^{(i)}_2=0$. For
$3\leq j\leq i$, set
$\widetilde{\beta}^{(i)}_j:=\widetilde{\beta}^{(i)}_{j-2}+\beta^{(i)}_j$.
Then $$r^{(i)}_j-e_j-\widetilde{\beta}^{(i)}_je_{j+1}\in{\rm
Span}\{e_{j+2},\ldots,e_{n-1}\}$$ for every $j \ :\ 1\leq j\leq i$.
\item[\rm (iv)] One has ${\rm Span}\{r^{(i)}_1,\ldots,r^{(i)}_i\}={\rm
Span}\{v_1,\ldots,v_{k+1},v_{k+2},\ldots, v_i\}$.
\item[\rm (v)] If
$\alpha_3=\ldots=\alpha_{k+1}=0$ then $v_i=e_i$ for any $1\leq
i\leq n-1$ and $r^{(i)}_j=e_j$ for any $k+2\leq i\leq n-1$ and
$1\leq j\leq i$.
\end{itemize}
\end{lemma}
\begin{proof}
We prove (i) by induction. For $i=k+1$ one has $r_j^{(k+1)}=v_j$
and $\beta^{(k+1)}_j=\alpha_j$ so that by Lemma \ref{lemma-4.6} (i), the
claim is true. Assume the claim is true for $i\geq k+1$ and show
for $i+1.$ For $j=3,4$ one has
$r_j^{(i+1)}=w(r_{j-2}^{(i)})+\beta_j^{(i+1)}w(r_{j-1}^{(i)})$
(since $\beta_j^{(i+1)}=0$) so that the claim is trivially true.
For $j\geq 5$ one has
$$\begin{array}{rcll}
r_j^{(i+1)}&=&w(r_{j-2}^{(i)})=
w(w(r^{(i)}_{j-4})+\beta_{j-2}^{(i)}w(r_{j-3}^{(i)}))&{\rm( by\ induction\ hypothesis)}\\
&=&w(r^{(i+1)}_{j-2}+ \beta_{j}^{(i+1)}r_{j-1}^{(i+1)})&{\rm (by\ definition)}\\
&=&w(r^{(i+1)}_{j-2})+\beta_{j}^{(i+1)}w(r_{j-1}^{(i+1)}).&\\
\end{array}
$$
Thus, the claim is true. This provides
$u(r_j^{(i)})=r^{(i)}_{j-2}+\beta_j^{(i)}r_{j-1}^{(i)}$ so that
${\rm Span}\{r^{(i)}_1,\ldots,r^{(i)}_{i}\}$ is $u$-stable.  To
show that ${\rm Span} \{r_1^{(i-1)},\ldots,r_{i-1}^{(i-1)}\}
\subset {\rm Span}\{r_1^{(i)},\ldots,r_i^{(i)}\}$ note that for
$3\leq j\leq i-1$ we have
$r_j^{(i-1)}=w(r_{j-2}^{(i-1)})+\beta_j^{(i-1)}w(r_{j-1}^{(i-1)})=r_{j}^{(i)}+\beta_j^{(i-1)}r_{j+1}^{(i)}$.

To show (ii) first note that $r_j^{(i)}=w(r^{(i-1)}_{j-2})$ and
$\beta_j^{(i)}=\beta_{j-2}^{(i-1)}$ so that the first three equalities follow by
a straightforward induction on $i\geq k+1.$ 
The proof of the last claim in (ii)
is exactly the same as the proof of the last claim in Lemma
\ref{lemma-4.6} (ii).

The proof of (iii) is exactly the same as the proof of Lemma
\ref{lemma-4.6} (iii).

To prove (iv) recall that $v_i=r_i^{(i)}$ for $i\geq k+2$ and note that by (i)
$${\rm Span}\{ v_1,\ldots,v_{k+1},r^{(k+2)}_{k+2},\ldots,
r^{(i)}_{i}\}\subset {\rm Span}\{
r^{(i)}_1,\ldots,r^{(i)}_{i}\}.$$ Further by Lemma \ref{lemma-4.6} (iii)
we have $v_j-e_j\in {\rm Span}\{e_{j+1},\ldots,e_{n-1}\}$ for any $j\leq
k+1$ and by (iii) we have $r^{(j)}_j-e_j\in{\rm Span}\{e_{j+1},\ldots,
e_{n-1}\}$ for any $j\geq k+2$. Therefore, the family of vectors
$\{v_1,\ldots,v_{k+1},r^{(k+2)}_{k+2},\ldots,r_i^{(i)}\}$ is
linearly independent so that
$$\dim {\rm Span}\{v_1,\ldots,v_{k+1},r^{(k+2)}_{k+2},\ldots,r_i^{(i)}\}=i\geq \dim {\rm Span}\{ r^{(i)}_1,\ldots,r^{(i)}_{i}\}$$
which provides the equality.

We prove (v) by induction on $i$. It is trivially true for
$i=1,2.$ Assume it is true for $2\leq i\leq k$ and show for $i+1.$
In that case $v_{i+1}=w(v_{i-1})=w(e_{i-1})=e_{i+1}.$ Now if it is true
for $i\geq k+1$ then $r_j^{(i+1)}=e_j$ for $j=1,2$ by definition.
For $3\leq j\leq i+1$ we have $r_{j}^{(i+1)}=w(r_{j-2}^{(i)})=w(e_{j-2})=e_j$.
In particular, $v_i=r_i^{(i)}=e_i$ for $k+2\leq i\leq n-1$.  
\end{proof}

\subsection{Nonsingularity of the component $\K^{Q(k,k,1)}$}

\label{4.8}

Now we are ready to show
\begin{proposition}
Every $F_{(d)}\ :\ 3\leq d\leq k+2$ lies in $\K^{Q(k,k,1)}$ and is a nonsingular point.
In particular, $\K^{Q(k,k,1)}$ is nonsingular.
\end{proposition}
\begin{proof}
As we have explained in \ref{4.3}, it is sufficient to construct  a
closed immersion $\Phi:\Co^{k+2}\rightarrow \Omega_{(d)}$
satisfying the two conditions:
\begin{itemize}
\item[(A)] For $(a_1,\ldots,a_{k+2})$ such that $a_i\ne 0$ for any
$i$, one has $\Phi(a_1,\ldots,a_{k+2})\in \F_{Q(k,k,1)}$.
\item[(B)] $\Phi(0,\ldots,0)=F_{(d)}$.
\end{itemize}

We start with the case $d=k+2.$ Consider
$(\alpha_1,\alpha_3,\ldots,\alpha_{k+1},\gamma_k,\gamma_{k+1})\in\Co^{k+2}$
and let $\{v_i\}_{i=1}^{n-1}$ be the vectors associated to
$\{\alpha_i\}_{i=3}^{k+1}$ constructed in \ref{4.6}.
Set $\{\gamma_i\}_{i=1}^{k-1}$ inductively by
$\gamma_i:=-\alpha_{i+2}\gamma_{i+1}$ for $i\geq 2$ and
$\gamma_1=-(\alpha_3-\alpha_1)\gamma_2.$

Now we define vectors $\eta_1,\ldots,\eta_n$ as
follows. Set $\eta_1:=e_1+\alpha_1e_2+\gamma_1e_n$,
$\eta_2:=e_2+\gamma_2e_n$ and $\eta_{k+2}:=e_n$. For $i\ :\ 3\leq
i\leq k+1$  set $\eta_i:=v_i+\gamma_ie_n$. For $i\ :\ k+3\leq
i\leq n$ set $\eta_i:=v_{i-1}$. For $i\geq 1$ put $U_i={\rm
Span}\{\eta_1,\ldots,\eta_i\}$.

Set
$\Phi(\alpha_1,\alpha_3,\ldots,\alpha_{k+1},\gamma_k,\gamma_{k+1}):=(U_0,\ldots,U_n)$.
Note that
\begin{itemize}
\item{} {\it $(U_0,\ldots,U_n)\in \Omega_{(d)}$ for any
$(\alpha_1,\alpha_3,\ldots,\alpha_{k+1},\gamma_k,\gamma_{k+1})\in\Co^{k+2}$.}
This is straightforward from Lemma \ref{lemma-4.6} (iii) and Lemma
\ref{lemma-4.7} (iii).

\item{}  $\Phi(0,\ldots,0)=F_{(d)}$ by Lemma
\ref{lemma-4.7} (v).

\item{} {\it $\Phi:\Co^{k+2}\rightarrow
\Omega_{(d)}$ is a closed immersion.} To show this recall from
\ref{4.3} that we have
$\eta_{i}=e_{(k+2)(i)}+\sum_{j=i+1}^n\phi_{i,j}e_{(k+2)(j)}$
Consider the dual morphism of algebras $$\Phi^*:\Co[\phi_{i,j}:1\leq i<j\leq
n]\rightarrow
\Co[\alpha_1,\alpha_3,\ldots,\alpha_{k+1},\gamma_k,\gamma_{k+1}].$$
 Let us show that $\Phi^*$ is surjective.
Indeed,
$\alpha_1,\alpha_3,\ldots,\alpha_{k+1},\gamma_k,\gamma_{k+1}$ can
be expressed in terms of the $\phi_{i,j}$'s. First we have
$\eta_1=e_1+\alpha_1e_2+\gamma_1e_n$ hence
$\alpha_1=\phi_{1,2}\in\mathrm{Im}\,\Phi^*$. For $i=3,\ldots,k+1$
by Lemma \ref{lemma-4.6} (iii) one has
$\eta_i=e_i+\widetilde{\alpha}_ie_{i+1}+\gamma_ie_n+\eta'_i$ where
$\eta'_i\in\mathrm{Span}\{e_{i+2},\ldots,e_{n-1}\}$. Hence
$\widetilde{\alpha}_i\in \mathrm{Im}\,\Phi^*$.
Thus,
$\alpha_i=\widetilde{\alpha}_i-\widetilde{\alpha}_{i-2}\in\mathrm{Im}\,\Phi^*$.
Note also that $\gamma_k=\phi_{k,k+2}$ and $\gamma_{k+1}=\phi_{k+1,k+2}$ so that $\gamma_k,\gamma_{k+1}\in
\mathrm{Im}\,\Phi^*$. Therefore, $\Phi^*$ is surjective so that
$\Phi$ is a closed immersion.

\item{} {\it
$(U_0,\ldots,U_n)\in\B_u$ for any
$(\alpha_1,\alpha_3,\ldots,\alpha_{k+1},\gamma_k,\gamma_{k+1})\in\Co^{k+2}$.}
Indeed,  by  definition of $v_i$ and $\eta_i$ we get $u(U_i)=0$
for $i=1,2$; also
$$\begin{array}{rcl}
U_i&=&{\rm Span}\{e_1+\alpha_1e_2+\gamma_1e_n,e_2+\gamma_2e_n,\ldots,v_{k+1}+\gamma_{k+1}e_n,e_n,v_{k+2},\ldots, v_{i-1}\}\\
&=&{\rm Span}\{e_n,v_1,\ldots,v_{i-1}\}\\
\end{array}$$
 for $i\geq k+2$ so that it is $u$-stable by
Lemma \ref{lemma-4.7} (i) and (iv). Further since
$\gamma_1=-(\alpha_3-\alpha_1)\gamma_2$ we get
$$\begin{array}{rl}u(\eta_3)=&u(v_3+\gamma_3e_n)=e_1+\alpha_3e_2=
(e_1+\alpha_1e_2+\gamma_1e_n)+(\alpha_3-\alpha_1)(e_2+\gamma_2e_n)\\
=&\eta_1+(\alpha_3-\alpha_1)\eta_2.\\
\end{array}$$
In the same way, taking into account
$\gamma_{i-2}=-\alpha_i\gamma_{i-1}$ for $i\geq 4$ we get for $i\
:\ 4\leq i\leq k+1$:
$$ u(\eta_i)=u(v_i+\gamma_ie_n)=v_{i-2}+\alpha_iv_{i-1}=(v_{i-2}+\gamma_{i-2}e_n)+\alpha_i(v_{i-1}+\gamma_{i-1}e_n)
=\eta_{i-2}+\alpha_i\eta_{i-1}$$ thus, $u(\eta_i)\in U_{i-1}$ for
any $3\leq i\leq k+1.$ Thus, $U_i$ is $u-$stable also for any
$3\leq i\leq k+1.$

\item{} {\it If $\alpha_3,\ldots,\alpha_k\ne 0$
then $(U_0,\ldots,U_n)\in \F_{Q(k,k,1)}$.} Indeed, by the
definition of $\eta_1,\eta_2$ one has $U_2\subset\ker u$ and if
$\alpha_3,\ldots,\alpha_k\ne 0$ then by Lemma \ref{lemma-4.6} (ii)
$\eta_i\not\in \ker u^{i-2}$ thus $J(u|_{U_i})=(i-1,1)$ for any
$2\leq i\leq k+1$. Then since $\eta_{k+2}=e_n$ we get
$J(u|_{U_{k+2}})=(k,1,1)$ and since $(U_0,\ldots,U_n)\in \B_u$
this provides  $(U_0,\ldots,U_n)\in \F_{Q(k,k,1)}.$
\end{itemize}
This completes the proof for $d=k+2.$

\medskip
The construction in the case $d<k+2$ is very similar.

Let $(\alpha_1,\alpha_3,\ldots,
\alpha_d,\alpha_{d+2},\ldots\alpha_{k+2},\gamma_{d-1},\nu)\in\Co^{k+2}$
and we set $\alpha_{d+1}:=-\nu\gamma_{d-1}.$ Let
$\{v_i\}_{i=1}^{n-1}$ be the set of vectors associated to
$\{\alpha_i\}_{i=3}^{k+1}$ as they are defined in \ref{4.6}.
Set $\{\gamma_i\}_{i=2}^{d-2}$ by
$\gamma_i=-\alpha_{i+2}\gamma_{i+1}$ and
$\gamma_1=-(\alpha_3-\alpha_1)\gamma_2$ (exactly as in the
previous case). And again as in the previous case we define
$\{\eta_i\}_{i=1}^n$ by $\eta_1:=e_1+\alpha_1e_2+\gamma_1e_n$,
$\eta_2:=e_2+\gamma_2e_n$ and for $3\leq i\leq d-1$ set
$\eta_i:=v_{i}+\gamma_ie_n.$ And exactly as in the previous time
for $i\geq k+3$ set $\eta_i:=v_{i-1}.$ We change the definition of
$\eta_i$ for $d\leq i\leq k+2.$  Put $\eta_d:=e_n+\nu v_d$ and
$\eta_{k+2}:=v_{k+1}.$ Finally, put
$\eta_i:=v_{i-1}+\alpha_{i+1}v_i$ for $d+1\leq i\leq k+1.$ And
again put $U_i:={\rm Span}\{\eta_1,\ldots,\eta_i\}$ and consider
the flag $\Phi(\alpha_1,\alpha_3,\ldots,
\alpha_d,\alpha_{d+2},\ldots\alpha_{k+2},\gamma_{d-1},\nu):=(U_0,U_1,\ldots,U_n)$.
Exactly as in the case $d=k+2$ one has
\begin{itemize}
\item{} {\it $(U_0,\ldots,U_n)\in \Omega_{(d)}$ for any
$(\alpha_1,\alpha_3,\ldots,
\alpha_d,\alpha_{d+2},\ldots\alpha_{k+2},\gamma_{d-1},\nu)\in\Co^{k+2}$.}
This is a straightforward consequence of Lemma \ref{lemma-4.6} (iii) and Lemma
\ref{lemma-4.7} (iii).

\item{}  $\Phi(0,\ldots,0)=F_{(d)}$ by Lemma
\ref{lemma-4.7} (v).

\item{} {\it $\Phi:\Co^{k+2}\rightarrow
\Omega_{(d)}$ is a closed immersion.} To show this we again
consider the dual map
$$\Phi^*:\Co[\phi_{i,j}\ :\ 1\leq i<j\leq n]\rar \Co[\alpha_1,\alpha_3,\ldots,\alpha_d,\alpha_{d+2},\ldots,\alpha_{k+2},\gamma_{d-1},\nu]$$
and show that it is surjective. Again, since
$\eta_1=e_1+\alpha_1e_2+\gamma_1e_n$ we get
$\alpha_1=\phi_{1,2}\in {\rm Im}\, \Phi^*$. 
Consider $\{\widetilde{\alpha}_i\}_{i=3}^{k+1}$ introduced in Lemma \ref{lemma-4.6} (iii). 
Set in addition $\widetilde{\alpha}_{k+2}:=\alpha_{k+2}+\tilde\alpha_k$.
For $3\leq i\leq d-1$
we get exactly as in case $d=k+2$ that
$\widetilde\alpha_i\in
{\rm Im}\, \Phi^*.$ Note also that $\gamma_{d-1}=\phi_{d-1,d}$.
Further, $\eta_d=e_n+\nu e_d+\eta'_d$ where $\eta'_d\in{\rm
Span}\{e_{d+1},\ldots,e_{n-1}\}$ so that $\nu=\phi_{d,d+1}.$ By
our definition $\alpha_{d+1}=-\nu\gamma_{d-1}$ so that
$\alpha_{d+1}\in {\rm Im}\, \Phi^*$ thus,
$\widetilde\alpha_{d+1}=\widetilde\alpha_{d-1}+\alpha_{d+1}\in{\rm
Im}\, \Phi^*$. For $d+2\leq i\leq k+2$ again one has
$\eta_{i-1}=e_{i-2}+(\widetilde\alpha_{i-2}+\alpha_i)e_{i-1}+\eta'_{i-1}$
where $\eta'_{i-1}\in{\rm Span}\{e_{i+2},\ldots,e_{n-1}\}$ so that
$\phi_{i-1,i}=\widetilde\alpha_{i-2}+\alpha_i=
\widetilde\alpha_i.$ Thus,
$\{\widetilde\alpha_i\}_{i=3}^{d-1}\cup\{\widetilde\alpha_i\}_{i=d+1}^{k+2}\subset
{\rm Im}\, \Phi^*.$ It remains to show that
$\widetilde\alpha_d\in {\rm Im}\, \Phi^*$ to get that
$\alpha_i=\widetilde\alpha_i-\widetilde\alpha_{i-2}\in {\rm Im}\,
\Phi^*$ for any $3\leq i\leq k+2$ which completes the proof. To
prove this, set $i=2k+2-d$. Then, since $3\leq d\leq k+1$, one has
$k+2\leq i+1\leq n.$ Thus,
$\eta_{i+1}=v_i=e_i+\widetilde\beta_i^{(i)}e_{i+1}+\eta'_{i+1}$
where $\eta'_{i+1}\in{\rm Span}\{e_{i},\ldots,e_{n-1}\}$. Thus,
$\phi_{i+1,i+2}=\widetilde\beta_i^{(i)}.$ By Lemma \ref{lemma-4.6} (iii)
and Lemma \ref{lemma-4.7} (ii)--(iii) we get
$\widetilde\beta_i^{(i)}=\widetilde\alpha_{i-2(i-k-1)}=\widetilde\alpha_d$
which completes the proof.

\item{} {\it $(U_0,\ldots,U_n)\in\B_u$ for any
$(\alpha_1,\alpha_3,\ldots,
\alpha_d,\alpha_{d+2},\ldots\alpha_{k+2},\gamma_{d-1},\nu)\in\Co^{k+2}$.}
For $i\geq k+2$, as in the case $d=k+2$, we can see that $U_i={\rm
Span}\{e_n,v_1,\ldots,v_{i-1}\}$ so that $U_i$ is $u$-stable.
Further, $U_i$ is $u$-stable for $i\leq d-1$  exactly as in the
case $d=k+2$. Further, since $\gamma_{d-2}=-\alpha_d\gamma_{d-1}$
we get exactly as in the case $d=k+2$:
$$u(\eta_d)=\nu u(v_d)=\nu(v_{d-2}+\alpha_dv_{d-1})=\nu(\eta_{d-2}+\alpha_d\eta_{d-1}).$$
Thus, $U_d$ is $u$-stable. We also note that by the same formula we
get $u(v_d)\in U_{d-1}.$ Since $\alpha_{d+1}=-\nu\gamma_{d-1}$ we
get
$$u(v_{d+1})=v_{d-1}+\alpha_{d+1}v_d=(v_{d-1}+\gamma_{d-1}e_n)-\gamma_{d-1}(e_n+\nu v_d)=\eta_{d-1}-\gamma_{d-1}\eta_d$$
so that $u(v_{d+1})\in U_d.$ Now for $i\ :\ d+2\leq i\leq k+1$ we
get $u(v_i)=v_{i-2}+\alpha_i v_{i-1}=\eta_{i-1}$ so that
$u(v_i)\in U_{i-1}.$ Altogether this provides us that for $i:
d+1\leq i\leq k+1$ one has
$$u(\eta_i)=u(v_{i-1}+\alpha_{i+1}v_i)\in U_{i-1}$$
so that $U_i$ is $u$-stable for any $d+1\leq i\leq k+1$, which
completes the proof.

\item{} {\it Finally, if
$\alpha_3,\ldots,\alpha_d,\alpha_{d+2},\ldots,\alpha_{k+2},\gamma_{d-1},\nu
\ne 0$ then $(U_0,\ldots,U_n)\in \F_{Q(k,k,1)}$.} First, note that
$\gamma_{d-1},\nu\ne 0$ implies $\alpha_{d+1}\ne 0.$ Since
$\eta_i$ is defined  exactly as in the case $d=k+2$ for $i\leq
d-1$ we get $J(u|_{U_i})=(i-1,1)$ for any $i\leq d-1$. Further,
since $\eta_d=e_n+\nu v_d$ and  $\eta_i=v_{i-1}+\alpha_{i+1}v_i$
for $d+1\leq i\leq k+1$  Lemma \ref{lemma-4.6} (ii) provides
$\eta_i\in\ker u^{i-1}$ and $\eta_i\not\in \ker u^{i-2}$ so that
again  $J(u|_{U_i})=(i-1,1)$ for any $d\leq i\leq k+1$. As we saw
in the previous  item $U_{k+2}={\rm
Span}\{e_n,v_1,\ldots,v_{k+1}\}$ so that $J(u|_{U_{k+2}})=(k,1,1)$
and since $(U_0,\ldots,U_n)\in \B_u$ this provides
$(U_0,\ldots,U_n)\in \F_{Q(k,k,1)}.$
\end{itemize}
This completes the proof for $d<k+2.$  
\end{proof}

\section*{Index of the notation}

\begin{itemize}
\item[\ref{1.1}\ ] $V$, $n$, $u$, ${\mathcal B}$, ${\mathcal
B}_n$, ${\mathcal B}_u$, $J(u)$, $Y(u)$, $Y_\lambda$,
$\mathbf{Tab}_\lambda$ \item[\ref{1.3}\ ] $\lambda^*$, $\mathrm{sh}\,(T)$,
$\pi_{1,i}(T)$, ${\mathcal F}_T$, ${\mathcal K}^T$
\item[\ref{2.2}\ ] $\pi_{i,j}(T)$, $Sch(T)$, ${\mathcal F}'_T$
\item[\ref{3.0}\ ] $\lambda_{[i,j]}$, $\varsigma_i$, $(T)_{i,j}$,
$T_j$, $T_{[i,j]}$, $St(T)$ \item[\ref{3.1}\ ] $C(T)$, $C^{-1}(S)$
\item[\ref{3.5}\ ] $C_{[k,l]}(T)$, $Sch_{[k,l]}(T)$
\item[\ref{3.6}\ ] $Eqs(T)$ \item[\ref{3.7}\ ] $P(r,s)$
\item[\ref{4.0}\ ] $Q(k,k,1)$, $P(m,m|t)$
\item[\ref{4.1}\ ] $r_T(i)$, $\tau(T)$, $j(T)$, $\mathrm{dist}(T)$
\item[\ref{4.3}\ ] $F_\sigma$,
$\mathbf{S}_u$, $(d)$, $F_{(d)}$,
$\Omega_{(d)}$, $(\eta_1,\ldots,\eta_n)$, $(\phi_{i,j})_{1\leq
i<j\leq n}$
\end{itemize}

\end{document}